\documentclass[11pt,a4paper]{amsart}
\usepackage{amssymb,color}
\usepackage[english]{babel}          
\usepackage{verbatim}
\usepackage[T1]{fontenc}             

\usepackage{floatflt,graphicx}
\usepackage{color,xcolor}
\usepackage{enumerate}
\usepackage{animate}
\usepackage{amsmath}
\usepackage{amsthm}
\usepackage{epstopdf}

\newcounter{minutes}
\divide\time by 60
\newcounter{hours}
\multiply\time by 60 \addtocounter{minutes}{-\time}
\title
[Lipschitz conditions, triangular ratio metric]
{Lipschitz conditions, triangular ratio metric, and quasiconformal maps }

\author{Jiaolong Chen}
\address{Department of Mathematics\\
Hunan Normal University\\ Changsha, China}
\email{  jiaolongchen@sina.com}

\author{Parisa Hariri}
\address{Department of Mathematics and Statistics\\
  University of Turku\\ Turku, Finland}
\email{parhar@utu.fi}

\author{Riku Kl\'en }
\address{Department of Mathematics and Statistics\\
  University of Turku\\ Turku, Finland\\
Massey University, Auckland, New Zealand}
\email{ripekl@utu.fi}

\author{Matti Vuorinen}
\address{Department of Mathematics and Statistics\\
  University of Turku\\ Turku, Finland}
\email{vuorinen@utu.fi}


\dedicatory{}

\theoremstyle{plain}
\newtheorem{thm}[equation]{Theorem}
\newtheorem{cor}[equation]{Corollary}
\newtheorem{lem}[equation]{Lemma}

\theoremstyle{definition}

\theoremstyle{remark}
\newtheorem{rem}[equation]{Remark}

\newtheorem{conj}[equation]{Conjecture}

\numberwithin{equation}{section}

\newtheorem{algor}[equation]{Algorithm}

\newcommand{\beq}{\begin{equation}}
\newcommand{\eeq}{\end{equation}}

\newcommand{\bequu}{\begin{eqnarray*}}
\newcommand{\eequu}{\end{eqnarray*}}

\newcommand{\bequ}{\begin{eqnarray}}
\newcommand{\eequ}{\end{eqnarray}}

\newcommand{\R}{\mathbb{R}^2}

\newcommand{\Bn}{ {\mathbb{B}^n} }
\newcommand{\Rn}{ {\mathbb{R}^n} }

\renewcommand{\tanh}{\,\textnormal{th}}

\begin{document}

\def\thefootnote{}
\footnotetext{ \texttt{File:~\jobname .tex,
           printed: \number\year-\number\month-\number\day,
           \thehours.\ifnum\theminutes<10{0}\fi\theminutes}
} \makeatletter\def\thefootnote{\@arabic\c@footnote}\makeatother

\begin{abstract} The triangular ratio metric is studied in subdomains of the
complex plane and Euclidean $n$-space. Various inequalities are proven for this metric.
The main results deal with the behavior of this metric under quasiconformal
maps. We also study the smoothness of metric disks with
small radii.
\end{abstract}

\maketitle
\section{Introduction}
\setcounter{equation}{0}

A significant part of geometric function theory deals with the behavior of distances under well known classes of mappings such as M\"obius transformations, bilipschitz maps or quasiconformal
mappings. Thus measurement of distances in terms of metrics is a common tool in function theory and frequently hyperbolic metrics or metrics of hyperbolic type are used in addition to Euclidean
or chordal distance. Many authors have contributed to this development in recent years. See for instance \cite{h}, \cite{himps}, \cite{kl}, \cite{pt}. A survey of these developments is given in
\cite{vu1}.

The triangular
ratio metric  is defined as follows for a domain $G \subsetneq   \mathbb{R}^n$ and
$x,y \in G$:
\begin{equation}\label{sm}
s_G(x,y)=\sup_{z\in \partial G}\frac{|x-y|}{|x-z|+|z-y|}\in [0,1].
\end{equation}
Clearly, the supremum in the definition \eqref{sm} of $s_G$ is attained at some
point $z \in \partial G\,,$ but finding this point is
 a nontrivial problem even for the case when $G$ is the unit disk.
P. H\"ast\"o \cite[Theorem 6.1]{h} proved that $s_G$ satisfies the triangle inequality and developed theory for metrics
more general than $s_G\,$ and generalized the work of A. Barrlund \cite{ba}. Very recently, the geometry of the balls of $s_G$
for some special domains was studied in \cite{hklv}.
Our goal here is to continue the study of this metric and to explore its behavior under M\"obius transformations, quasiconformal and quasiregular mappings. We also give upper and lower bounds for this metric in terms of other metrics in several domains such as the unit ball, the upper half plane and $\mathbb{R}^n\setminus\{0\}$, the whole space $\mathbb{R}^n$ punctured at the origin. Also some ideas for further work are pointed out.

The paper is divided into sections as follows. In Section 2 we
give  algorithms for numerically finding the value of $s_G(x,y)$, for instance,
in the case of a domain bounded by a polygon. In Section 3 we develop
the main ideas of this paper and relate the triangular ratio metric to other well-known metrics of geometric function theory such as the hyperbolic metric of the unit ball or half-space or to the distance ratio metric of a domain $G\subset\mathbb{R}^n$. In Section 5 we apply these results and well-known distortion results of quasiconformal maps to study how the triangular ratio metric behaves under quasiconformal and quasiregular mappings. In Section 4 we study the smoothness of the boundaries of $s-$disks in a triangle and in a rectangle.
We now proceed to formulate some of our main results.

\begin{thm}\label{1m}
\begin{enumerate}
\item
Let $f:\mathbb{H}^n\rightarrow \mathbb{H}^n$ be a $K-$quasiregular mapping. Then for $x,y\in \mathbb{H}^n$ we have
\[
s_{\mathbb{H}^n}(f(x),f(y))\leq\lambda_n^{1-\alpha}(s_{\mathbb{H}^n}(x,y))^{\alpha},~ \alpha=K^{1/(1-n)}\,,
\]
where $\lambda_n \in [4,2 e^{n-1}), \lambda_2=4,$ is the Gr\"otzsch ring constant depending only on $n$ (\cite[Lemma 7.22]{vu}).

\medskip
\item
Let $f:\mathbb{B}^n\rightarrow \mathbb{B}^n$ be a $K-$quasiregular mapping.
Then for $x,y\in \mathbb{B}^n$ we have
\[
s_{\mathbb{B}^n}(f(x),f(y))\leq 2^{\alpha}\lambda_n^{1-\alpha}(s_{\mathbb{B}^n}(x,y))^{\alpha},~ \alpha=K^{1/(1-n)}.
\]
\medskip
\item
Let $f:\mathbb{B}^n\rightarrow \mathbb{H}^n$ be a $K-$quasiregular mapping.
Then for $x,y\in \mathbb{B}^n$ we have
\[
s_{\mathbb{H}^n}(f(x),f(y))\leq 2^{\alpha}\lambda_n^{1-\alpha}(s_{\mathbb{B}^n}(x,y))^{\alpha},~ \alpha=K^{1/(1-n)}.
\]
\medskip
\item
Let $f:\mathbb{H}^n\rightarrow \mathbb{B}^n$ be a $K-$quasiregular mapping.
Then for $x,y\in \mathbb{H}^n$ we have
\[
s_{\mathbb{B}^n}(f(x),f(y))\leq \lambda_n^{1-\alpha}(s_{\mathbb{H}^n}(x,y))^{\alpha},~ \alpha=K^{1/(1-n)}.
\]

\end{enumerate}
\end{thm}
\medskip

\begin{thm}\label{2m}
Let $G=\mathbb{R}^n\setminus\{0\},$ and $f:G \to  G$ be a $K-$quasiconformal mapping with $f(\infty)=\infty,$
and let $z,w$ be two distinct points in $G$ and $ \alpha= K^{1/(1-n)}\,.$ Then
\[
s_{fG}(f(z),f(w))\leq \frac{1}{P_5(n,K)}\left(s_G(z,w)\right)^\alpha,\quad s_G(z,w)=\frac{|z-w|}{|z|+|w|},
\]
 where $P_5(n,K)\to 1, K \to 1$, and $P_5(n,K)$ is defined in Lemma \ref{29}.
\end{thm}

Of particular interest is the special
case $K=1$ of Theorems \ref{1m} and \ref{2m}. Clearly, Theorem \ref{2m} is
sharp in this case and the same is true
about Theorem \ref{1m} (1).
The question about the best constant in Theorem \ref{1m} (2) deserves some attention for the case when $K=1=\alpha$.
The constant on the right hand side is then $2$.

For a detailed study of this constant we define for $a\in (0,1)$ the class $C(a)$ of all M\"obius transformations $h: \mathbb{B}^n\rightarrow \mathbb{B}^n$ with
$|h(0)|=a$ and the constant
\begin{equation}
L(a)=\sup\{s_{\mathbb{B}^n}(h(x),h(y))/s_{\mathbb{B}^n}(x,y):~x,y\in\mathbb{B}^n, x \neq y, h\in C(a)\}.
\end{equation}

\medskip

\begin{thm}\label{mob}
For $n=2\,,$ $L(a)\geq 1+a.$
\end{thm}

\medskip

Theorem \ref{mob} shows that for $K=1$ the constant $2$ in Theorem \ref{1m} (2) cannot be replaced by a smaller constant (independent of $a$).
\medskip

\begin{conj}\label{conj1.9}
Our numerical experiments for $n=2$ suggest that $L(a)=1+a$.
\end{conj}
\medskip

In Theorem \ref{3.36} we show that $\frac{1-a}{1+a}\leq L(a)\leq\frac{1+a}{1-a}$.

\medskip

For a domain $G\subset \mathbb{R}^n, x,y\in G,$ we define the $j$-metric by
$$j_G(x,y)=\log\left(1+\frac{|x-y|}{\min\{d_G(x),d_G(y)\}}\right),$$
where $d_G(z)=d(z,\partial G)$. We will omit the subscript $G$ if it is clear from context. This metric has found numerous applications
in geometric function theory, see \cite{himps,vu}. We also define
$$p_G(x,y)=\frac{|x-y|}{\sqrt{|x-y|^2+4\,d_G(x)\,d_G(y)}}.$$

We next formulate some of our comparison results between metrics.

\begin{thm} \label{mainA}
  Let $G$ be a proper subdomain of $\mathbb{R}^n$. Then for all $x,y \in G$ we have
  \[
    p_G(x,y)\leq \frac{1}{\sqrt{2}} j_G(x,y),
  \]
and
  \[
    s_G(x,y) \leq \frac{1}{\log 3}j_G(x,y),
  \]
where the constant $\frac{1}{\log 3} \approx 0.91$ is the best possible.
\end{thm}

\begin{thm}\label{mainB}
  (1) Let $t \in (0,1)$ and $m \in \{ j, p, s \}$. There exists a constant $c_m=c_m(t) > 1$ such that for all $x,y \in \mathbb{B}^n$ with $|x|,|y|<t$ we have
  \[
    m_{\mathbb{B}^n}(x,y) \leq c_m m_{\mathbb{R}^n \setminus \{ e_1 \}}(x,y).
  \]
  Moreover, $c_m(t) \to 1$ as $t \to 0$ and $c_m(t) \to \infty$ as $t \to 1$, for all $m\in\{j,p,s\}$.

  \noindent (2) Let $G \subset \mathbb{R}^n$, $x \in G$, $t \in (0,1)$ and $m \in \{ j, p, s \}$. Then there exists a constant $c_m = c_m(t)$ such that for all $y,z \in G \setminus \mathbb{B}^n(x,t d_G(x))$ we have
  \[
    m_{G \setminus \{ x \}}(y,z) \leq c_m m_G(y,z).
  \]
  Moreover, the constant is best possible as $t \to 1$. This means that $c_j, c_p, c_s \to 2$ as $t \to 1$.
\end{thm}

We also study the geometry of balls of the $s$-metric. We use the notation
$$ B_{s_{G}}(x,r)=\{z\in G: s_G(x,z)<r\} $$
for the balls of the $s$-metric. First we
show, for $n=2\,,$ that disks of small enough radii have smooth boundaries and our
main result here is Theorem \ref{thm1.8}.

Let us denote $T_{\frac{\pi}{6},2}$ the equilateral triangle with vertices $(0,0)$, $(\sqrt{3},1)$, $(\sqrt{3},-1)$, and $R_{a,b}$ the rectangle with vertices $(a,b)$, $(a,-b)$, $(-a,b)$, $(-a,-b)$, where $a\geq b>0$.

\begin{thm}\label{thm1.8}
(1) Let  $G=T_{\frac{\pi}{6},2}$, $x=(x_{1},x_{2})\in G$, $r>0$. Then the metric ball $ B_{s_{G}}(x,r) $ is smooth if and only if $r\leq r_{0}$ or $r\leq r_{1}$, where
\begin{center}
$r_{0}=\min \left \{\frac{2|x_{2}|}{|x|},\frac{|x_2|-\sqrt{3}x_1+2 }{\sqrt{(x_1-\sqrt{3})^{2}+(1-|x_2|)^{2}}} \right\},$ and
$r_{1}=\frac{\sqrt{3}x_1-2-|x_{2}|}{\sqrt{(x_1-\sqrt{3})^{2}+(1-|x_2|)^{2}}} .$
\end{center}

(2) Let  $G=R_{a,b}$, $x=(x_{1},x_{2})\in G$, $r>0$. Then the metric ball $ B_{s_{G}}(x,r) $ is smooth if and only if $r\leq r_{2}$ or $r\leq r_{3}$, where
\begin{center}
$r_{2}=\min \left \{\frac{|x_{2}|}{b},\frac{(a-|x_1|)-(b-|x_2|)}{\sqrt{(a-|x_1|)^{2}+(b-|x_2|)^{2}}} \right\},$ and
$r_{3}=\min \left \{\frac{|x_{1}|}{a},\frac{(b-|x_2|)-(a-|x_1|)}{\sqrt{(a-|x_1|)^{2}+(b-|x_2|)^{2}}} \right\}.$
\end{center}

\end{thm}






\section{Algorithms for numerical computation of $s_G$}

The hyperbolic metric $\rho_{\mathbb{H}^n}$ and $\rho_{\mathbb{B}^n}$ of the upper
half plane ${\mathbb{H}^n} = \{ (x_1,\ldots,x_n)\in {\mathbb{R}^n}:  x_n>0 \} $
and of the unit ball ${\mathbb{B}^n}= \{ z\in {\mathbb{R}^n}: |z|<1 \} $ can be defined as weighted metrics with the weight functions
 $w_{\mathbb{H}^n}(x)=1/{x_n}$ and
 $w_{\mathbb{B}^n}(x)=2/(1-|x|^2)\,,$ respectively. This definition as such
is rather abstract and for applications concrete formulas
are needed. By \cite[p.35]{b} we have
\begin{equation}\label{cro}
\cosh{\rho_{\mathbb{H}^n}(x,y)}=1+\frac{|x-y|^2}{2x_ny_n}
\end{equation}
for all $x,y\in \mathbb{H}^n$, and by \cite[p.40]{b} we have
\begin{equation}\label{sro}
\sinh{\frac{\rho_{\mathbb{B}^n}(x,y)}{2}}=\frac{|x-y|}{\sqrt{1-|x|^2}{\sqrt{1-|y|^2}}}
\end{equation}
and
\begin{eqnarray}\label{tro}
\tanh{\frac{\rho_{\mathbb{B}^n}(x,y)}{2}}&=&\frac{|x-y|}{\sqrt{|x-y|^2+(1-|x|^2)(1-|y|^2)}}\\
&=&\frac{|x-y|}{|x||x^*-y|},\, x^*=\frac{x}{|x|^2},\nonumber
\end{eqnarray}
for all $x,y\in\mathbb{B}^n\setminus\{0\}$. As shown in \cite[Theorem 4.2]{hklv} we have
\begin{equation} \label{s_H0}
s_{\mathbb{H}^n}(x,y)= \tanh {\frac{\rho_{\mathbb{H}^n}(x,y)}{2}}=\frac{|x-y|}{|x-\bar{y}|},
\end{equation}
for all $x,y\in\mathbb{H}^n$, where $\bar{y}$ is the reflection of $y$ with respect to $ \partial{\mathbb H}^n$. See also \eqref{s_H} below. Unfortunately,
there is no formula similar to \eqref{s_H0} for the case of $s_{\mathbb{B}^n}\,.$
Therefore inequalities for  $s_{\mathbb{B}^n}\,$ are needed, see Section 3 below.

Explicit formulas for $s_G(x,y)$ are known only for a few particular cases. Our goal is to list several domains
for which we have written algorithms in the MATLAB language. The definition of $s_G(x,y)$ readily shows that the supremum is attained and that a point $z\in\partial{G}$ with $s_G(x,y)=\frac{|x-y|}{|x-z|+|z-y|}$ is located on
the maximal ellipse with foci $x$ and $y$ and contained in $\overline{G}$. The point $z$ is called an extremal point. Finding this maximal ellipse is however a difficult task even for $\mathbb{B}^2$.
In the course of this research we have extensively made use of experiments using the algorithms in this section. In particular, Conjecture \ref{conj1.9} is based on these algorithms.
\begin{algor}{\bf $s_{\mathbb{B}^2}$}
\end{algor}
Let $x,y\in\mathbb{B}^2$ and $z\in\partial{\mathbb{B}^2}$ be such that
\begin{equation}
s_{\mathbb{B}^2}(x,y)=\frac{|x-y|}{|x-z|+|z-y|}.
\end{equation}
The point $z$ can be found by choosing $m$ equally spaced points on the smaller arc on $\partial{\mathbb{B}^2}$ between $\frac{x}{|x|}$ and $\frac{y}{|y|}$ and selecting the point $z$ that minimizes the expression $|x-z|+|z-y|$ among the chosen points, say for $m= 1000$.
\begin{figure}[h]
\begin{center}
     \includegraphics[width=10cm]{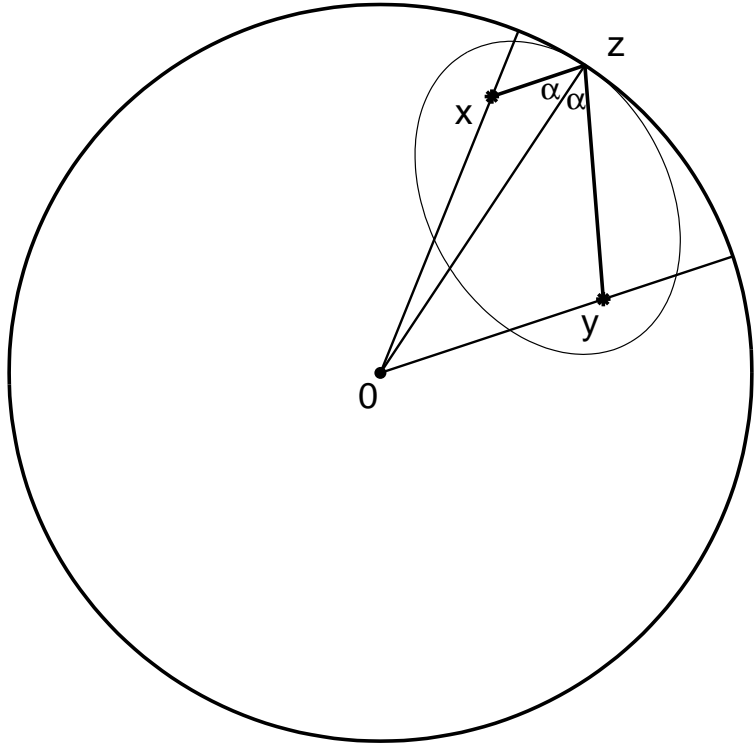}
\caption{The maximal ellipse with foci $x$ and $y$ and contained in $\overline{\mathbb{B}^2}$. }
\end{center}
    \end{figure}

\begin{algor}{\bf  $s_{\mathbb{H}^2}$}
\end{algor}
Suppose that $x,y\in\mathbb{H}^2$ are two distinct points. An extremal point $z\in\partial{\mathbb{H}^2}=\mathbb{R}$ for $s_{\mathbb{H}^2}(x,y)$ minimizes
the sum $$|x-z|+|z-y|=|x-z|+|z-\bar{y}|,$$
where $\bar{y}$ is as in formula \eqref{s_H0}. Therefore $z$ is the unique point of intersection of the segment $[x,\bar{y}]$ with the real axis.
In conclusion,
\begin{equation} \label{s_H}
s_{\mathbb{H}^2}(x,y)=\frac{|x-y|}{|x-\bar{y}|}.
\end{equation}

\begin{rem}\label{2.10}
Sometimes it is convenient to write the formula \eqref{s_H} in a different form which we give now. Suppose that $x, y\in\mathbb{H}^2$ with $d(x)=d(x,\partial{\mathbb{H}^2})\geq d(y)$. Let $w$ be a point on the segment
$[x,\overline{x}]\cap \mathbb{H}^2$ with $d(w)=d(y)$ and $\alpha=\measuredangle (w-y,x-y)$. Then clearly $|w-y|=|\overline{w}-\overline{y}|=|x-y|\cos{\alpha}$ and $\sin{\alpha}=(d(x)-d(y))/|x-y|$. This yields
\[
|y-w|=\sqrt{|x-y|^2-(d(x)-d(y))^2}
\]
and also, by the Pythagorean Theorem,
\[
|x-\overline{y}|^2=|x-\overline{w}|^2+|\overline{w}-\overline{y}|^2=|x-y|^2+4d(x)d(y)
\]
In conclusion, the formula \eqref{s_H} can also be written as
\[
s_{\mathbb{H}^2}(x,y)=\frac{|x-y|}{\sqrt{|x-y|^2+4d(x)d(y)}}.
\]
Therefore we see by \eqref{s_H0} that
\[
p_{\mathbb{H}^2}(x,y)=s_{\mathbb{H}^2}(x,y)=\tanh{\frac{\rho_{\mathbb{H}^2}(x,y)}{2}}.
\]
\end{rem}

\begin{algor}{\bf   $s_R$, $R$ is a rectangle}
\end{algor}
Given distinct $x,y$ in a rectangle $R$, the extremal boundary point $z$ as in \eqref{sm} must be located on one of the four sides $T_j\, ,j=1,\cdots ,4$ of $R\,.$
 If $y_j$ is the reflection point of $y$ with respect to side $T_j\, ,j=1,\ldots,4$, then $z_j=[x,y_j]\cap\partial{R}$ and
\begin{equation}
s_R(x,y)=\frac{|x-y|}{\min \{|x-y_j|: j=1,2,3,4\}}.
\end{equation}

\begin{algor}{\bf  $s_A$, $A$ is a sector}
\end{algor}
Let $\alpha\in (0,\pi)$ and $A=\{z\in\mathbb{C}:0<\arg{z}<\alpha\}$. Given $x,y\in A$,
the extremal point $z\in\partial{A}$ for $s_A(x,y)$ has only two options: it
is located either on the real axis $\{x\in\mathbb{R}: x \geq 0 \}$ or on the ray
$\{t \exp{i\alpha}: t>0 \}$.
In the first case by \eqref{s_H}

$$s_A(x,y)=\frac{|x-y|}{|x-\bar{y}|},$$
whereas in the second case again by \eqref{s_H}
$$s_A(x,y)=\frac{|x-y|}{|x-y_2|},$$
where $y_2=|y|\exp{i(2\alpha-\arg{y})}$. In conclusion, in both cases
\begin{equation}
s_A(x,y)=\frac{|x-y|}{\min\{|x-\bar{y}|,|x-y_2|\}}.
\end{equation}
This idea can be extended in a straightforward way to
triangles and other convex polygons.

\begin{algor}{\bf  $s_P$, $P$ is a polygon}
\end{algor}
Suppose that $v_1,v_2,\ldots,v_m $ are points in the plane such that the polygon with these points as vertices
is a bounded Jordan domain. 
The method is based on exhaustive tabulation of function values and choosing the optimal point on $\partial P\,.$
We parameterize $\partial{P}$ using the polygonal curve length as a parameter, measured from $v_1$ via
the points $v_j$. Then this real parameter varies on $[0,L]$ where $$L=\Sigma_{j=1}^{m}|v_j-v_{j+1}|,$$
and we agree that $v_{m+1}=v_1$. The parametrization $z:[0,L]\longrightarrow \partial{P}$ enables us to find all
the competing points for the definition of $s_P(x,y)$. Then finding $s_P(x,y)$ becomes a $1-$dimensional minimization problem,
which can be solved by exhaustive tabulation.
%

\section{Comparison results for $s_G$}\label{section3}

The goal of this section is to find inequalities between distances of points in terms of simple expressions. Problems of two kinds are considered. First, if $m_G$ is a metric defined in a domain $G\subset\Rn$, $x, y\in G$, then we compare $m_G(x,y)$ and $m_{G_1}(x,y)$ where $G_1$ is a simple domain. Second, if we have two metrics $e_G$ and $d_G$ on a domain $G\subset \Rn$, then we estimate $e_G(x,y)$ in terms of $d_G(x,y)$. In several cases, this comparison is carried out not in the whole domain but in $B^n(x_0,\lambda d(x_0,\partial{G}))$ where $x_0\in G$ is a fixed point and $\lambda\in (0,1)$ is a constant. In some results we consider the case of $\mathbb{B}^n$. Some examples of the metrics we use are the hyperbolic and distance ratio metrics and the $s$ and $v$ metrics.

From the definition \eqref{sm} of $s_G$ it is clear that
$s_G$ has three important properties:
\begin{itemize}
\item [(a)]  {\it monotonicity with respect
to domain}, i.e. if $D_1,D_2 \subset {\mathbb R}^n $ are domains with
$D_1 \subset D_2$ and $x,y\in D_1\,,$ then $s_{D_1}(x,y)  \ge s_{D_2}(x,y).$
\item [(b)] {\it Sensitivity to boundary variation,} i.e.
if $D \subset {\mathbb R}^n $ is a domain and $x_0 \in D\,,$ then
the numerical values of $s_D(x,y)$ and $s_{D\setminus \{x_0\}}(x,y)$ are not
comparable if $x,y$ are very close to $x_0\,.$
\item [(c)] For fixed $x,y\in G,$ {\it one extremal boundary point $z \in \partial G$ determines
the numerical value of $s_G(x,y)\,.$}
\end{itemize}

Our goal is to find various inequalities for $s_G$
in terms of expressions that are explicit. In particular, we hope to get
rid of the supremum in \eqref{sm}, and hope to use expressions that have the
above properties (a)--(c). Most of these expressions define metrics and we
will show that these metrics are locally quantitatively equivalent.

For a domain $G\subset \mathbb{R}^n, x,y\in G,$ we define {\it the visual angle metric} \cite{klvw}
by
$$
v_G(x,y) = \sup  \{ \measuredangle(x,z,y): z \in \partial G  \}\,.
$$
The metrics $j_G$, $v_G$ and $s_G$ have the aforementioned three properties (a)--(c) and $p_G \le 1,
v_G \le \pi$ while
$j_G$ is unbounded. All of the expressions $s_G, v_G, j_G, p_G$ are invariant
under similarity transformations.

\medskip

\begin{rem}
Because the inequality $p_{\mathbb{B}^2}(t,0)+p_{\mathbb{B}^2}(0,-t)>p_{\mathbb{B}^2}(t,-t)$, fails for small $t\,,$ we see that $p_G$ is not a metric.
\end{rem}

\medskip

\begin{lem}\cite[  Lemma 2.41(2)]{vu}, \cite[ Lemma 7.56]{avv}\label{10} Let $G \in \{  {\mathbb B}^n , {\mathbb H}^n\}\,,$ and let
 $\rho_G$ stand for the respective hyperbolic metric. Then for all $x,y\in G$
 $$ j_G(x,y)\le \rho_G(x,y) \le 2j_G(x,y).$$
\end{lem}

\medskip

The following theorem solves a question posed in \cite[Open problem 3.2]{hklv}.

\begin{thm} \label{riku}
  Let $G$ be a proper subdomain of $\mathbb{R}^n$. Then for all $x,y \in G$ we have
  \[
    s_G(x,y) \leq \frac{1}{\log 3}j_G(x,y)
  \]
  and the constant $\frac{1}{\log 3} \approx 0.91$ is the best possible.
\end{thm}

\begin{proof}
  Let us fix the points $x$ and $y$. By rescaling the domain we may assume that $|x-y|=1$. We can also assume that $d(x) \leq d(y)$, because otherwise we can swap the points.

  We denote $t=d(x)>0$. Now
  \[
    j_G(x,y) = \log \left( 1+\frac{1}{t} \right)
  \]
  and we divide the proof into two cases: $t \le \frac{1}{2}$ and $t > \frac{1}{2}$.

  We assume first that $t \le \frac{1}{2}$. Now $j_G(x,y) \geq \log 3$ and since $s_G(x,y) \leq 1$ we have
  \[
    s_G(x,y) \leq 1 \leq \frac{j_G(x,y)}{\log 3}.
  \]

  We assume then that $t > \frac{1}{2}$. We want to maximize $s_G(x,y)$ in terms of $t$. 
Now
  \[
    s_G(x,y) \le s_{\mathbb{B}^n(x,t) \cup \mathbb{B}^n(y,t)}(x,y) = \frac{|x-y|}{|x-z|+|y-z|} = \frac{1}{2t}
    \]
  $z\in \partial{(\mathbb{B}^n(x,t) \cup \mathbb{B}^n(y,t))}$, and we want to find a lower bound for the function
  \[
    f(t) = \frac{j_G(x,y)}{s_G(x,t)} \ge 2t \log \left( 1+\frac1t \right), \quad t>\frac12.
  \]
  We can show that $g(t)=\frac{\log(1+t)}{t}$ is decreasing for $t$, because
  \begin{eqnarray*}
  g'(t)&=&\frac{\frac{t}{1+t}-\log(1+t)}{t^2}\\
  &\leq& \frac{\frac{t}{1+t}-\frac{2t}{2+t}}{t^2} \leq 0 \\
  \end{eqnarray*}
  so it is increasing for $\frac1t$, thus
  $f(t)$ is increasing. We collect $f(t) > f(\frac12) = \log 3$ and the claimed inequality is proved.

  The constant $\frac{1}{\log 3}$ can be easily verified to be the best possible by investigating the domain $G=\mathbb{R}^n \setminus \{ 0 \}$. For any $x \in G$ selecting $y=-x$ gives $s_G(x,y) = 1$ and $j_G(x,y)=\log 3$.
\end{proof}

\begin{lem}\label{sp}
(1) If $x,y\in G \subset\mathbb{R}^n$ and $G$ is convex, then
\[
 s_G(x,y)\leq p_G(x,y).
\]
Here equality holds for
all $x,y \in G $ if $ G= \mathbb{H}^n\,.$

(2) For $x,y\in G \subset\mathbb{R}^n$,
\[
p_G(x,y)\leq \sqrt{2} s_G(x,y).
\]
\end{lem}
\begin{proof}
(1) Suppose that $z\in \partial G$ is an extremal boundary point for the $s$-metric for which
the equality holds in \eqref{sm}. We draw a line $L$ through $z$, tangent to $\partial{G} \,.$ Let $\overline{y}$ be the reflection of $y$ in the line $L$. By geometry, see Remark \ref{2.10},
$$|x-z|+|z-y|=|x-\overline{y}|=\sqrt{|x-y|^2+4 d_1(x)d_1(y)},$$
$d_1(x)=d(x,L)$, $d_1(y)=d(y,L)\,.$
Because $G$ is convex it is clear that $L$ is outside $G$, but $d(x), d(y)$ are the shortest distances from $x,y$ to $\partial{G}$
, so obviously $d(x)\leq d_1(x)$, $d(y)\leq d_1(y)$, thus
\begin{eqnarray}
 s_G(x,y) &=& \frac{|x-y|}{|x-z|+|z-y|}      \nonumber \\
   &=& \frac{|x-y|}{\sqrt{|x-y|^2+4 d_1(x)d_1(y)}} \nonumber \\
   &\leq & \frac{|x-y|}{\sqrt{|x-y|^2+4 d(x)d(y)}} \nonumber \\
   &=& p_G(x,y).\nonumber
\end{eqnarray}
(2) Fix $x, y\in G$, $z\in\partial{G}$, such that $d(x)=|x-z|$. By symmetry we may assume $d(x)\leq d(y)$ and then
\beq\label{3s}
s_G(x,y)\geq \frac{|x-y|}{|x-y|+2d(x)}.
\eeq
Now by \cite[1.58 (13)]{avv} and \eqref{3s}
\bequu
p_G(x,y) &\leq& \frac{|x-y|}{\sqrt{|x-y|^2+4d(x)^2}}\\
&\leq & \frac{|x-y|}{2^{1/2-1}(|x-y|+2d(x))}\\
&\leq & \frac{\sqrt{2}|x-y|}{|x-y|+2d(x)}\leq \sqrt{2} s_G(x,y).
\eequu
\qedhere
\end{proof}
It is easy to see that convexity cannot be omitted from Lemma \ref{sp} (1). For instance if $G=\R\setminus\{0\}$ and $x=(0,1)=-y$, then the inequality in Lemma \ref{sp} (1) fails.
\begin{lem}
For $x,y \in \mathbb{B}^n$ we have
$$s_{\mathbb{B}^n}(x,y) \geq s_{\mathbb{B}^n}(x_s,y_s) = \frac{|x-y|}{\sqrt{|x-y|^2 + 4(1-|m|)^2}},$$
where $m=\frac{x_0+y_0}{2}$ and $x_0,y_0\in\partial{{\mathbb{B}^n}}$ are the points of intersection of the line through $x$ and $y $ with $\partial{{\mathbb{B}^n}},$  $|x-y|=|x_s-y_s|\,,$  $|x_s|=|y_s|$
and further
$$|m|=\frac{\sqrt{|x|^2|y|^2-(x\cdot y)^2}}{|x-y|},$$
$$x_s=x_0+\frac{y_0-x_0}{|y_0-x_0|}(\sqrt{1-|m|^2}-\frac{|x-y|}{2}),$$
$$y_s=y_0+\frac{x_0-y_0}{|x_0-y_0|}(\sqrt{1-|m|^2}-\frac{|x-y|}{2}),$$
and hence
$$s_{\mathbb{B}^n}(x,y) \geq \frac{|x-y|^2}{|x-y|^4+4(|x-y|-\sqrt{|x|^2|y|^2-(x\cdot y)^2})^2}.$$
\end{lem}
\begin{proof}
If we move $x, y \in \mathbb{B}^n$ to $x_s, y_s \in \mathbb{B}^n$ which are symmetric with respect to midpoint $m$ of the segment $[x_0,y_0]$, then we see easily that
the extremal ellipse with foci $x_s, y_s$ is larger than the extremal ellipse with foci $x, y$ and hence by \eqref{sm},
$$s_{\mathbb{B}^n}(x,y) \geq s_{\mathbb{B}^n}(x_s,y_s) = \frac{|x-y|}{\sqrt{|x-y|^2 + 4(1-|m|)^2}}.$$
Here $|m|$ is the shortest distance from the origin to the line $\overline{xy}$, which by the Law of Cosines, $|m|=\frac{\sqrt{|x|^2|y|^2-(x\cdot y)^2}}{|x-y|},$
and therefore
 $$s_{\mathbb{B}^n}(x_s,y_s) = \frac{|x-y|^2}{|x-y|^4+4(|x-y|-\sqrt{|x|^2|y|^2-(x\cdot y)^2})^2},$$
 and the proof is complete.
\end{proof}
\begin{lem}
For $x, y \in \mathbb{B}^n$ with $|x|>|y|$, $y_r= x- \frac{x}{|x|}|x-y|= - \frac{x}{|x|}(|x|-|x-y|)$,

$$s_{\mathbb{B}^n}(x,y)\geq s_{\mathbb{B}^n}(x,y_r)  = \frac{|x-y|}{|x-y|+2(1-t)} \equiv w(x,y),~ t=\max \{ |x|, |y|\}.$$

\end{lem}
\begin{proof}
Note that $y_r\in[x,-x]$ and $|x-y|=|x-y_r|$. By geometric properties of the ellipse it is clear that $s_{\mathbb{B}^n}(x,y)\geq s_{\mathbb{B}^n}(x,y_r)$ and thus
\begin{eqnarray}
s_{\mathbb{B}^n}(x,y)&=& \sup_{z\in \partial G}\frac{|x-y|}{|x-z|+|z-y|}\nonumber\\
&\geq& s_{\mathbb{B}^n}(x,y_r)\nonumber\\
& =& \frac{|x-y|}{|x-y|+2(1-t)},~ t=\max \{ |x|, |y|\}.\nonumber
\end{eqnarray}
\qedhere
\end{proof}

\medskip

\begin{lem}\label{3.7}
For all $x, y\in\mathbb{B}^n$ we have
\begin{equation}
p_{\mathbb{B}^n}(x,y) \leq\tanh{\frac{\rho_{\mathbb{B}^n}(x,y)}{2} }\leq 2p_{\mathbb{B}^n}(x,y).
\end{equation}
\end{lem}
\begin{proof}
The second inequality follows from Lemma \ref{sp} and Theorem \ref{rk2}. For the first inequality clearly
\begin{eqnarray}
(1-|x|^2)(1-|y|^2) &=& (1-|x|)(1-|y|)(1+|x|)(1+|y|)      \nonumber \\
   &\leq & 4(1-|x|)(1-|y|) ,\nonumber
\end{eqnarray}
so
\begin{eqnarray*}
 \tanh{\frac{\rho_{\mathbb{B}^n}(x,y)}{2} }&=& \frac{|x-y|}{\sqrt{|x-y|^2+(1-|x|^2)(1-|y|^2)}}     \\
   &\geq & \frac{|x-y|}{\sqrt{|x-y|^2+4d(x)d(y)}}  \\
   &= & p_{\mathbb{B}^n}(x,y).
\end{eqnarray*}
\qedhere
\end{proof}

\begin{thm}
If $z\in G$, $0<\lambda<1$, $x,y\in\mathbb{B}^n(z,\lambda d(z))$, then
\begin{equation}\label{sj}
s_{\mathbb{B}^n(z,d(z))}(x,y)\leq C j_{\mathbb{B}^n(z,d(z))}(x,y),~ C=\frac{2(1-\lambda)}{1+2\lambda},
\end{equation}

\begin{equation}\label{js}
j_{\mathbb{B}^n(z,d(z))}(x,y)\leq \frac{2(1+\lambda)}{1-\lambda}s_{\mathbb{B}^n(z,d(z))}(x,y).
\end{equation}

\end{thm}
\begin{proof}
From $x,y\in\mathbb{B}^n(z,\lambda d(z))$ it follows that
\begin{equation}\label{help}
\frac{|x-y|}{d(z)}\leq 2\lambda.
\end{equation}
 Because for all $x, y\in\mathbb{B}^n(z,\lambda d(z))$, $w\in\partial{G}$, the inequality
 $$|x-w|+|y-w|\geq 2(1-\lambda)d(z),$$ holds, we see that
$$s_{\mathbb{B}^n(z,d(z))}(x,y)\leq \frac{|x-y|}{2(1-\lambda)d(z)},$$
and by $\log(1+t)\geq \frac{2t}{2+t},$ for $t\geq 0$, and \eqref{help} we see that
\begin{eqnarray}
j_{\mathbb{B}^n(z,d(z))}(x,y) &\geq & \log\left(1+\frac{|x-y|}{(1+\lambda)(d(z))}\right)\nonumber \\
&\geq&\frac{\frac{2|x-y|}{(1+\lambda)(d(z))}}{2+\frac{|x-y|}{(1+\lambda)(d(z))}} \nonumber \\
&\geq& \frac{|x-y|}{(1+2\lambda)(d(z))}. \nonumber
\end{eqnarray}
Hence it suffices to choose $C=\frac{2(1-\lambda)}{1+2\lambda}.$

\medskip

For the second part observing that for $w\in\mathbb{B}^n(z,\lambda d(z))\,,$ $d(w)\geq (1-\lambda)d(z)$ we have
\begin{eqnarray}
j_{\mathbb{B}^n(z,d(z))}(x,y)&\leq& \log\left(1+\frac{|x-y|}{(1-\lambda)(d(z))}\right)\nonumber \\
&\leq & \frac{|x-y|}{(1-\lambda)(d(z))}.\nonumber
\end{eqnarray}
On the other hand, setting $w=z+d(z)\frac{y-z}{|y-z|}$ we see that
\begin{eqnarray*}
|x-w|+|y-w| &\leq & |x-y|+|y-w|+|y-w|\\
&\leq & |x-y|+2d(z)\leq 2(1+\lambda)d(z)
\end{eqnarray*}
and hence
\[
s_{\mathbb{B}^n(z,d(z))}(x,y)\geq  \frac{|x-y|}{2(1+\lambda)d(z)}.
\]
Now it suffices to find $C$ such that $$\frac{|x-y|}{2(1+\lambda)d(z)}\geq C\frac{|x-y|}{(1-\lambda)(d(z))},$$ so we may choose $C=\frac{2(1+\lambda)}{1-\lambda}$, and the proof is complete.
\end{proof}

\begin{thm}\label{jp}
If $z\in G$, $0<\lambda<1$, $x,y\in\mathbb{B}^n(z,\lambda d(z))$, then
$$j_G(x,y)\leq C p_G(x,y),~ C= \frac{2}{1-\lambda}.$$
\end{thm}
\begin{proof}
By symmetry we may assume that $d(x)\leq d(y)$. Then by $\log(1+t)\leq t,~ t>0$ we have
$$j_G(x,y)\leq\frac{|x-y|}{\min\{d(x),d(y)\}}=\frac{|x-y|}{d(x)}.$$
On the other hand by the assumption we get $d(z)\leq\frac{1}{1-\lambda}\min\{d(x),d(y)\},$ and $$\frac{1-\lambda}{1+\lambda}\leq\frac{d(x)}{d(y)}\leq\frac{1+\lambda}{1-\lambda},$$
\begin{eqnarray}
p_G(x,y)&=& \frac{|x-y|}{\sqrt{|x-y|^2+4 d(x)d(y)}}\nonumber \\
&\geq & \frac{|x-y|}{\sqrt{\left(2\lambda\frac{d(x)}{1-\lambda}\right)^2+4d(x)\frac{1+\lambda}{1-\lambda}d(x)}},\nonumber\\
&\geq & \frac{1-\lambda}{2}.\frac{|x-y|}{d(x)}.\nonumber
\end{eqnarray}
We see that
$$j_G(x,y)\leq\frac{|x-y|}{d(x)}\leq C\frac{1-\lambda}{2}.\frac{|x-y|}{2}\leq Cp_G(x,y),$$
holds if
$C\geq\frac{2}{1-\lambda},$ and the proof is complete.
\end{proof}

\begin{thm}\label{pj}
If $x, y\in G\subset\mathbb{R}^n$, then
$$p_G(x,y)\leq \frac{1}{\sqrt{2}} j_G(x,y).$$
\end{thm}
\begin{proof}
Fix $x, y\in G$. By relabeling the points we may assume that $d(x)\leq d(y)$. Then
\[
p_G(x,y)\leq \frac{|x-y|}{\sqrt{|x-y|^2+4 d(x)^2}},
\]
and
\[
j_G(x,y)=\log\left(1+\frac{|x-y|}{d(x)}\right).
\]
Write $t=|x-y|/d(x)$ and observe that
\[
j_G(x,y)=\log(1+t)\geq \frac{2t}{2+t}
\]
\[
p_G(x,y)\leq \frac{t}{\sqrt{t^2+4}}
\]
It is enough to find a constant $C$ such that
\[
\frac{2t}{2+t}\geq C \frac{t}{\sqrt{t^2+4}}
\]
for all $t\geq 0$. Easy calculation shows that we can choose $C=\sqrt{2}$.
\end{proof}
\begin{cor}
If $x, y\in G\subset\mathbb{R}^n$, and $G$ is convex, then
\[
s_G(x,y)\leq \frac{1}{\sqrt{2}} j_G(x,y).
\]
\end{cor}
\begin{proof}
It follows from Lemma \ref{sp} (1) and Theorem \ref{pj}.
\end{proof}
{\bf Proof of Theorem \ref{mainA}}
  The result follows from Theorems \ref{riku} and \ref{pj}. $\square$

\begin{thm}\label{jv2}
\begin{enumerate}
\item
For $x, y\in\mathbb{B}^2$ we have $$v_{\mathbb{B}^2}(x,y)\leq 2 j_{\mathbb{B}^2}(x,y).$$

\item
If $\lambda\in (0,1)$ and $x, y\in\mathbb{B}^2(\lambda)$ then $$\frac{3(1-\lambda^2)}{2(3+\lambda^2)}j_{\mathbb{B}^2}(x,y)\leq v_{\mathbb{B}^2}(x,y).$$
\end{enumerate}
\end{thm}
\begin{proof}

(1) By \cite[3.12]{klvw} we have $v_{\mathbb{B}^2}(x,y)\leq \rho_{\mathbb{B}^2}(x,y)\,.$ Now the proof follows by Lemma \ref{10}.

(2) By Lemma \ref{10} $$\sinh{\frac{\rho_{\mathbb{B}^2}(x,y)}{2}}\leq\sinh{j_{\mathbb{B}^2}(x,y)}\leq \sinh\left(\log\left(1+\frac{2\lambda}{1-\lambda}\right)\right)=\frac{2\lambda}{1-\lambda^2},$$
and by \cite[3.15]{klvw} $\rho_{\mathbb{B}^2}^{*}\leq v_{\mathbb{B}^2}\leq 2\rho_{\mathbb{B}^2}^{*},$ where
$$\rho_{\mathbb{B}^2}^{*}(x,y)=\arctan\left(\sinh{\frac{\rho_{\mathbb{B}^2}(x,y)}{2}}\right).$$
Next by \cite[1.8]{dc} $$\frac{3t}{1+2\sqrt{1+t^2}}<\arctan{t}<\frac{2t}{1+\sqrt{1+t^2}},$$
for $t>0$. We further obtain
\begin{eqnarray}
\rho_{\mathbb{B}^2}^{*}(x,y)&=&\arctan\left(\sinh{\frac{\rho_{\mathbb{B}^2}(x,y)}{2}}\right)\nonumber \\
&\geq& \frac{3\sinh{\frac{\rho_{\mathbb{B}^2}(x,y)}{2}}}{1+2\sqrt{1+\sinh^{2}{\frac{\rho_{\mathbb{B}^2}(x,y)}{2}}}}\nonumber \\
&\geq& \frac{3\sinh{\frac{j_{\mathbb{B}^2}(x,y)}{2}}}{1+2\sqrt{1+\left(\frac{2\lambda}{1-\lambda^2}\right)^2}}\nonumber \\
&=& \frac{3(1-\lambda^2)}{3+\lambda^2}\sinh{\frac{j_{\mathbb{B}^2}(x,y)}{2}}\nonumber \\
&\geq& \frac{3(1-\lambda^2)}{2(3+\lambda^2)}j_{\mathbb{B}^2}(x,y).\nonumber
\end{eqnarray}
Thus $$\frac{3(1-\lambda^2)}{2(3+\lambda^2)}j_{\mathbb{B}^2}(x,y)\leq v_{\mathbb{B}^2}(x,y).$$
\qedhere
\end{proof}

\begin{thm}
If $z\in G$, $\lambda\in (0,1)$ then for $x, y\in\mathbb{B}^n(z,\lambda d(z)),$
\begin{equation}\label{31m}
s_G(x,y)\leq\left(\frac{1+\lambda}{1-\lambda}\right) p_G(x,y).
\end{equation}

\end{thm}
\begin{proof}
By monotonicity of $s$-metric and Lemma \ref{sp} (1)
$$s_G(x,y)\leq s_{\mathbb{B}^n(z,d(z))}(x,y)\leq p_{\mathbb{B}^n(z,d(z))}(x,y)\leq \frac{|x-y|}{\sqrt{|x-y|^2+4(1-\lambda)^2 d(z)^2}}.$$
If $x,y\in \mathbb{B}^n(z,\lambda d(z))$, we easily see that
\begin{equation}\label{ine1}
(1-\lambda)d(z)\leq d_G(x)\leq (1+\lambda)d(z).
\end{equation}
Now if we choose $c=\left(\frac{1+\lambda}{1-\lambda}\right)$, then
\[
\frac{|x-y|}{\sqrt{|x-y|^2+4(1-\lambda)^2 d(z)^2}}\leq \frac{c|x-y|}{\sqrt{|x-y|^2+4(1+\lambda)^2 d(z)^2}}\leq c p_G(x,y).\qedhere
\]
\end{proof}

\begin{thm}\label{sv}
Let $0<\lambda<1$, $x,y\in\mathbb{B}^2(\lambda )$. Then
\begin{enumerate}
\item
$$s_{\mathbb{B}^2}(x,y)\leq \frac{4(3+\lambda^2)}{3(1+2\lambda)(1+\lambda)}v_{\mathbb{B}^2}(x,y),$$
\item
$$v_{\mathbb{B}^2}(x,y)\leq \frac{4(1+\lambda)}{1-\lambda}s_{\mathbb{B}^2}(x,y).$$
\end{enumerate}
\end{thm}
\begin{proof}

(1) By Theorem \ref{jv2} and \eqref{sj}, $$s_{\mathbb{B}^2}(x,y)\leq \frac{4(3+\lambda^2)}{3(1+2\lambda)(1+\lambda)}v_{\mathbb{B}^2}(x,y). $$

(2) By Theorem \ref{jv2} and \eqref{js},
\[
  v_{\mathbb{B}^2}(x,y)\leq 2j_{\mathbb{B}^2}(x,y)\leq \frac{4(1+\lambda)}{1-\lambda}s_{\mathbb{B}^2}(x,y).\qedhere
\]

\end{proof}

\begin{thm}\label{pv}
\begin{enumerate}
\item
If $\lambda\in (0,1)$ and $x, y\in\mathbb{B}^2(\lambda)$ then $$v_{\mathbb{B}^2}(x,y)\leq \frac{4(1+\lambda)}{(1-\lambda)} p_{\mathbb{B}^2}(x,y).$$

\item
If $x, y\in\mathbb{B}^2$ with $v_{\mathbb{B}^2}(x,y)\in (0,\pi/2),$ then $$p_{\mathbb{B}^2}(x,y)\leq  v_{\mathbb{B}^2}(x,y),$$
\end{enumerate}

\end{thm}
\begin{proof}
(1) By Theorems \ref{sv} and \ref{sp},
$$v_{\mathbb{B}^2}(x,y)\leq \frac{4(1+\lambda)}{(1-\lambda)} s_{\mathbb{B}^2}(x,y)\leq \frac{4(1+\lambda)}{(1-\lambda)} p_{\mathbb{B}^2}(x,y).$$

(2) By Lemma \ref{3.7} and \cite[3.15]{klvw} we have $$\rho_{\mathbb{B}^2}^{*}(x,y)=\arctan\left(\sinh{\frac{\rho_{\mathbb{B}^2}(x,y)}{2}}\right)\leq v_{\mathbb{B}^2}(x,y).$$
Then
$$\rho_{\mathbb{B}^2}(x,y)\leq 2{\rm arsinh}(\tan(v_{\mathbb{B}^2}(x,y))).$$
Then if $v_{\mathbb{B}^2}(x,y)\in (0,\pi/2),$
\begin{eqnarray}
p_{\mathbb{B}^2}(x,y) &\leq& \tanh({\rm arsinh}(\tan(v_{\mathbb{B}^2}(x,y))))\nonumber \\
&=&\frac{\tan(v_{\mathbb{B}^2}(x,y))}{\sqrt{1+\tan^2(v_{\mathbb{B}^2}(x,y))}}\nonumber \\
&=& \sin(v_{\mathbb{B}^2}(x,y))\nonumber\\
&\leq &  v_{\mathbb{B}^2}(x,y).\nonumber
\end{eqnarray} \qedhere
\end{proof}

\begin{thm}\label{rk2}
For $x, y\in\mathbb{B}^n$ we have
\begin{equation}
\tanh \left( \frac{\rho_{\mathbb{B}^n}(x,y)}{2} \right) \leq 2 s_{\mathbb{B}^n}(x,y).
\end{equation}
\end{thm}

\begin{proof}
Suppose first that one of the points $x$ and $y$ is $0$. Without loss of generality, we may suppose that $y=0$. From the definition of $s_{\mathbb{B}^n}$ it follows that for $z=\frac{x}{|x|}$
\[
s_{\mathbb{B}^n}(x,0)\geq \frac{|x-0|}{|x-z|+|z-0|}=\frac{|x|}{2-|x|}.
\]
Because
\[
\tanh \left( \frac{\rho_{\mathbb{B}^n}(x,y)}{2} \right)=|x|
\]
we easily see that the claim holds if one of the points is $0$. The case when both points are $0$ is trivial.

By \eqref{tro} and \eqref{sm} it is enough to show that $$I\leq 2|x||x^*-y|,\quad I=\inf_{z\in\partial{\mathbb{B}^n}} |x-z|+|z-y|,$$
Assume $|y|\leq |x|$. Denote $|y|=t|x|$ for $t\in [0,1]$, $\gamma\in [0,\pi],$ is angle between $[0,x]$ and $[0,y]$.

{\it Case A.} $\gamma\geq\frac{\pi}{2}.$
Now
\begin{equation}\label{rhs1}
2|x||x^*-y|\geq 2|x|\frac{1}{|x|}=2,
\end{equation}
Moreover choose $z_1=\frac{x}{|x|}$, then
\begin{eqnarray}\label{lhs1}
I&\leq& |x-z_1|+|z_1-y|\\
&\leq & 1-|x|+\sqrt{t^2|x|^2+1+2t|x|}\nonumber\\
&=& 2-|x|+t|x|\nonumber\\
&=& 2-(|x|(1-t))\leq 2.\nonumber
\end{eqnarray}
So by \eqref{rhs1} and \eqref{lhs1}, $$I\leq 2|x||x^*-y|,$$
{\it Case B.} $\gamma\leq\frac{\pi}{2}.$
\begin{equation}
2|x||x^*-y|=2||y|x-z_2| = 2||x|y-z_1|,
\end{equation}
where $|z_2|=\frac{y}{|y|}$ and $|z_1|=\frac{x}{|x|}$. Next we choose $z$ in the infimum to be the middle point of $z_1$ and $z_2$ on the unit sphere. This means that
$\measuredangle(x,0,z)=\measuredangle(z,0,y)=\gamma/2$ and $|z|=1$.
We know that
$$I\leq |x-z|+|z-y|,$$

\begin{figure}[h]
\begin{center}
     \includegraphics[width=10cm]{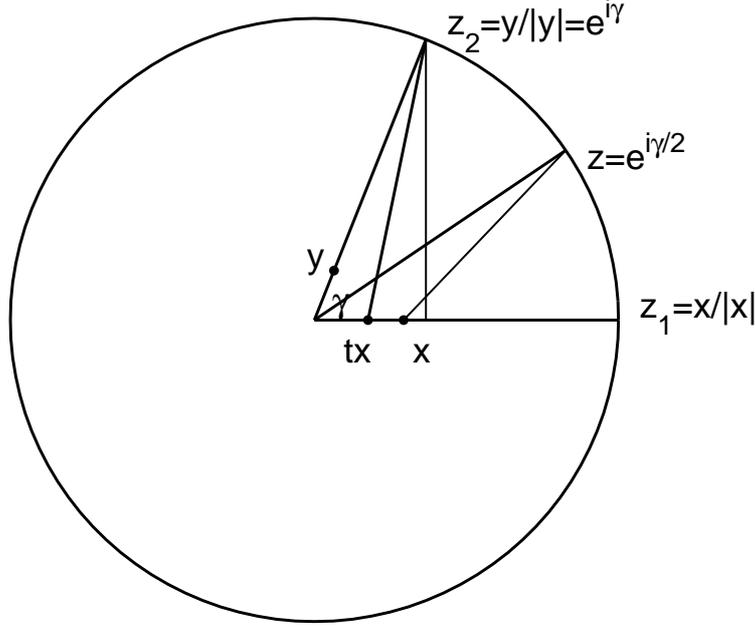}
\caption{Proof of Theorem \ref{rk2}. The case $r=|z-x|> \sin(\gamma)$. }
\end{center}
    \end{figure}
We next show that
\begin{equation}\label{mat}
p/r\geq 1,\quad p=|z_2-|y|x|,\quad r=|z-x|.
\end{equation}

By elementary geometry, applying the properties of the right triangle $\Delta(0, z_2, (\cos \gamma) z_1)$
and the Law of Cosines, we see that  
\begin{equation}\label{eleg}
p\geq |z_2-  (\cos \gamma) z_1|=\sin\gamma\geq \sqrt{1+\cos^2(\gamma)-2\cos(\gamma)\cos(\gamma/2)} =|z-(\cos\gamma)z_1| \,.
\end{equation}
The second inequality follows because for $\gamma\in (0,\pi/2),$
$$\sin^2(\gamma)>1+\cos^2(\gamma)-2\cos(\gamma)\cos(\gamma/2)$$
by basic trigonometry.

If $r\le \sin\gamma$, then by \eqref{eleg} $p/r\geq 1$ clearly holds.
In the remaining case $r=|z-x|> \sin\gamma$. Because $x \in [0,z_1]$,
this means by \eqref{eleg} that $x \in[0,  (\cos\gamma) z_1]$ and hence the angle between the segments $[x,z_2]$ and $[x,0]$ is more than $\pi/2\,$ and hence
$$ p=|z_2-|y|x|>|z_2-x|.$$

Finally, we see that $p/r\geq |z_2-x|/|z-x|>1$, because $x$ and $z$ both are in the same half plane determined by the bisecting normal of the segment $[z_2,z].$
Symmetrically we obtain that $$|z-y|\le ||x|y-z_1|,\,$$
and hence
$$|x-z|+|z-y|\le ||y|x-z_2|+||x|y-z_1|=2|x||x^*-y|$$
and the proof is complete.
\end{proof}
\begin{cor}\label{cor3.30}
\begin{enumerate}
\item
If $f:\mathbb{H}^n\rightarrow \mathbb{H}^n$ is a  M\"obius transformation onto $\mathbb{H}^n$, then for all $x,y\in\mathbb{H}^n$,
\[
s_{\mathbb{H}^n}(f(x),f(y))=s_{\mathbb{H}^n}(x,y).
\]
\item
If $f:\mathbb{H}^n\rightarrow \mathbb{B}^n$ is a  M\"obius transformation onto $\mathbb{B}^n$, then for all $x,y\in\mathbb{H}^n$,
\[
s_{\mathbb{B}^n}(f(x),f(y))\leq s_{\mathbb{H}^n}(x,y).
\]
\item
If $f:\mathbb{B}^n\rightarrow \mathbb{H}^n$ is a  M\"obius transformation onto $\mathbb{H}^n$, then for all $x,y\in\mathbb{B}^n$,
\[
s_{\mathbb{H}^n}(f(x),f(y))\leq 2s_{\mathbb{B}^n}(x,y).
\]
\item
If $f:\mathbb{B}^n\rightarrow \mathbb{B}^n$ is a  M\"obius transformation onto $\mathbb{B}^n$, then for all $x,y\in\mathbb{B}^n$,
\[
s_{\mathbb{B}^n}(f(x),f(y))\leq 2 s_{\mathbb{B}^n}(x,y).
\]
\end{enumerate}
\end{cor}
\begin{proof}
It is a basic fact that a M\"obius transformation $f:G\rightarrow D=fG$ with $G, D\in \{\mathbb{B}^n,\mathbb{H}^n\}$ defines an isometry $f:(G,\rho_G)\rightarrow (D,\rho_D)$ between hyperbolic spaces. This fact combined with \eqref{s_H0}, Lemma \ref{3.7} and Theorem \ref{rk2} yields the proof.
\end{proof}

\medskip
We were led to Conjecture \ref{conj1.9} by MATLAB experiments. We now show that if the conjecture holds true, then the constant
$1+a$ cannot be improved when $n=2\,.$

\medskip
{\bf Proof of Theorem \ref{mob}.}
Let $h(z)=\frac{z+a}{1+az}$. Then $h(0)=a$, $a>0$. Choose $b$ such that $h(b)=\frac{1+va}{1+v}$, $v>0$. Easy calculation yields $b=\frac{1}{1+v(1+a)}$. Since $s_{\mathbb{B}^2}(r,t)=\frac{t-r}{2-t-r}$ for $0<r<t$ we see that
\bequu
\frac{s_{\mathbb{B}^2}(h(0),h(b))}{s_{\mathbb{B}^2}(0,b)} &=& \frac{\frac{1+va}{1+v}-a}{2-a-\frac{1+va}{1+v}}\cdot \frac{2-\frac{1}{1+v(1+a)}}{\frac{1}{1+v(1+a)}}\\
&=& \frac{1+2v(1+a)}{1+2v}\rightarrow 1+a,
\eequu
when $v\rightarrow \infty. \hfill \square$
\begin{thm}\label{3.36}
 If $f : {\mathbb{B}^n} \to  {\mathbb{B}^n} = f( {\mathbb{B}^n})$ is a
M\"obius transformation with $f(a) =0$, for some $a\in \mathbb{B}^n$, then
  for all distinct points $x,y\in  {\mathbb{B}^n}$ we have
\[
\frac{1-|a|}{1+|a|} \, s_{\mathbb{B}^n}(x,y)\leq  s_{\mathbb{B}^n}(f(x),f(y)) \le \frac{1+|a|}{1-|a|} \, s_{\mathbb{B}^n}(x,y)\,.
\]
\end{thm}
  \begin{proof} If $f(0) =0$ then $f$ is a rotation and there is nothing
to prove. Otherwise $f(a) =0$ some
  $a \neq 0\,.$
  Let $f= T_a$ be the canonical representation of a M\"obius
transformation, see \cite{b}. Then
  with $a^*= a/|a|^2, r= \sqrt{|a|^{-2}-1}$ we have
  $$
  |T_a(x)-T_a(y)|= \frac{r^2 |x-y|}{|x-a^*||y-a^*|} \,.
  $$
  If $ w \in \partial  {\mathbb{B}^n}\,,$ then this formula yields
  \begin{eqnarray*}
  Q(x,y,w) & = & \frac{|T_a x - T_a y|}{|T_a x - T_a w| + |T_a w - T_a y|}
:  \frac{|x-y|}{|x-w|+|w-y|}\\
  & = & \frac{|x-w|+|w-y|}{\beta |x-w|+ \gamma|w-y|}
  \end{eqnarray*}
  with $\beta= |y-a^*|/|w-a^*|, \gamma= |x-a^*|/|w-a^*|\,.$ Clearly,
  $$
  |w-a^*| \le 1 +|a|^{-1}  \, \quad |x-a^*|, |y-a^*|\ge |a|^{-1}-1
  $$
  and hence
  $$
  Q(x,y,w)\le \frac{|x-w|+|w-y|}{ |x-w|+ |w-y|} \frac{1+|a|}{1-|a|} =
\frac{1+|a|}{1-|a|} \,.
  $$
  Thus we have for all $x,y\in   {\mathbb{B}^n}, w \in  \partial
{\mathbb{B}^n}$
  $$
   \frac{|T_a x - T_a y|}{|T_a x - T_a w| + |T_a w - T_a y|} \le
\frac{1+|a|}{1-|a|} \frac{|x-y|}{|x-w|+|w-y|}\,.
  $$
  Taking supremum over all $ w \in  \partial  {\mathbb{B}^n}$ yields the
second inequality. Because the inverse of a M\"obius transformation also is a M\"obius transformation, the first inequality follows from the second one.
  \end{proof}

We compare next $j$, $p$, $s$ and $v$ in domains $\mathbb{R}^n \setminus \{ e_1 \}$ and $\mathbb{B}^n$. By the monotonicity with respect to domains it is clear that for all $x,y \in \mathbb{B}^n$ and  $m \in \{ j, p, s, v \}$ we have $m_{\mathbb{R}^n \setminus \{ e_1 \}}(x,y) \leq m_{\mathbb{B}^n}(x,y)$. Next we consider the comparison in the opposite direction. Let us start by introducing the following lemma.

\begin{lem}\label{loglemma}
   For $0 < b \leq a$ the function
   \[
     f(x) = \frac{\log (1+ax)}{\log (1+bx)}, x \in (0,\infty),
   \]
   is decreasing.
\end{lem}
\begin{proof}
   Since
   \[
     f'(x) = \frac{\frac{a}{1+ax}\log (1+bx)-\frac{b}{1+bx}\log
(1+ax)}{\log^2 (1+bx)}
   \]
   the inequality $f'(x) \leq 0$ is equivalent to
   \begin{equation}\label{function f}
     \frac{1+bx}{b}\log(1+bx) \leq \frac{1+ax}{a}\log(1+ax).
   \end{equation}
   Now we show that the function
   \[
     g(c) = \frac{1+cx}{c}\log(1+cx)
   \]
   is increasing on $(0,\infty)$, which implies \eqref{function f} and
the assertion. This is clear because $0 < b \leq a$ and
   \[
     g'(c) = \frac{cx-\log(1+cx)}{c^2}>0
   \]
as $\log(1+y) < y$ for $y>0$.
\end{proof}

\begin{thm}\label{riku3}
  Let $t \in (0,1)$ and $m \in \{ j, p, s \}$. There exists a constant $c_m=c_m(t) > 1$ such that for all $x,y \in \mathbb{B}^n$ with $|x|,|y|<t$ we have
  \[
    m_{\mathbb{B}^n}(x,y) \leq c_m m_{\mathbb{R}^n \setminus \{ e_1 \}}(x,y).
  \]
  Moreover, $c_m(t) \to 1$ as $t \to 0$ and $c_m(t) \to \infty$ as $t \to 1$, for all $m\in \{j,p,s\}$.
\end{thm}
\begin{proof}
  We denote $m_1 = m_{\mathbb{B}^n}$, $m_2 = m_{\mathbb{R}^n \setminus \{ e_1 \} }$ and find upper bound for $\frac{m_1}{m_2}$, which gives us $c_m$.

  Let us start with $m=j$. We denote $z = |x-y| \in [0,2t)$ and obtain by Lemma \ref{loglemma}
  \begin{eqnarray*}
    \frac{j_1}{j_2} & = & \frac{\log \left( 1+ \frac{z}{\min \{ 1-|x|,1-|y| \}} \right) }{\log \left( 1+ \frac{z}{\min \{ |x-e_1|,|y-e_1| \}} \right) } \leq \frac{\log \left( 1+ \frac{z}{1-t} \right) }{\log \left( 1+ \frac{z}{1+t} \right) }\\
    & \leq & \lim_{z \to 0} \frac{\log \left( 1+ \frac{z}{1-t} \right) }{\log \left( 1+ \frac{z}{1+t} \right) } = \lim_{z \to 0} \frac{1+t+z}{1-t+z} = \frac{1+t}{1-t} = c_j,
  \end{eqnarray*}
    where the second equality follows from l'H\^{o}spital's rule. Obviously $c_j \to 1$ as $t \to 0$ and $c_j \to \infty$ as $t \to 1$

  Let us now consider $m=p$. Now
  \[
    \frac{p_1^2}{p_2^2} = \frac{|x-y|^2 + 4|x-e_1||y-e_1|}{|x-y|^2 + 4(1-|x|)(1-|y|)} \le \frac{4t^2 + 4(1+t)^2}{0+4(1-t)^2} = \frac{2t^2+2t+1}{t^2-2t+1}
  \]
  and we can choose
  \[
    c_p = \sqrt{ \frac{2t^2+2t+1}{t^2-2t+1} }.
  \]
  Clearly $c_p \to 1$ as $t \to 0$ and $c_p \to \infty$ as $t \to 1$.

  Next we set $m=s$ and obtain by geometry
  \[
    \frac{s_1}{s_2} = \frac{|x-e_1|+|y-e_1|}{\inf_{z \in \partial \mathbb{B}^n} |x-z|+|z-y|} \le \frac{2(1+t)}{2(1-t)} = \frac{1+t}{1-t} = c_s.
  \]
   Again it is clear that $c_s \to 1$ as $t \to 0$ and $c_s \to \infty$ as $t \to 1$.
\end{proof}

   Note that for the visual angle metric $v$ the result of Theorem \ref{riku3} does not hold. We would need an upper bound for
   \[
     \frac{v_{\mathbb{B}^n}(x,y)}{v_{\mathbb{R}^n \setminus \{ e_1 \}}(x,y)} = \frac{\sup_{z \in \partial \mathbb{B}^n} \measuredangle(x,z,y)}{\measuredangle(x,e_1,y)},
   \]
  but choosing $x$ and $y$ to be distinct points on the $x_1$-axis $$\sup_{z \in \partial \mathbb{B}^n} \measuredangle(x,z,y) > 0$$ and $\measuredangle(x,e_1,y) = 0$.

  Next result demonstrates the sensitivity to boundary variation. We consider domains $G \subset \mathbb{R}^n$ and $G' = G \setminus \{ x \}$, where $x \in G$. Again by the monotonicity we have $m_{G}(y,z) \leq m_{G'}(y,z)$ for all $y,z \in G'$ and  $m \in \{ j, p, s, v \}$.

\begin{thm}\label{riku4}
  Let $G \subset \mathbb{R}^n$, $x \in G$, $t \in (0,1)$ and $m \in \{ j, p, s \}$. Then there exists a constant $c_m = c_m(t)$ such that for all $y,z \in G \setminus \mathbb{B}^n(x,t d_G(x))$ we have
  \[
    m_{G \setminus \{ x \}}(y,z) \leq c_m m_G(y,z).
  \]
  Moreover, the constant is best possible as $t \to 1$. This means that $c_j, c_p, c_s \to 2$ as $t \to 1$.
\end{thm}
\begin{proof}
  We denote $G' = G \setminus \{ x \}$ and will find an upper bound for $\frac{m_{G'}(y,z)}{m_G(y,z)}$.

  We consider first the case $m=j$. If $d_G(y) = d_{G'}(y)$ and $d_G(z) = d_{G'}(z)$, then there is nothing to prove as $j_{G'}(y,z) = j_G(y,z)$ and we can choose $c_j=1$. We consider next two cases: $d_G(y) \neq d_{G'}(y)$, $d_G(z) = d_{G'}(z)$ and $d_G(y) \neq d_{G'}(y)$, $d_G(z) \neq d_{G'}(z)$.

  Let us assume $d_G(y) \neq d_{G'}(y)$ and $d_G(z) = d_{G'}(z)$ (or by symmetry we could as well assume $d_G(y) = d_{G'}(y)$ and $d_G(z) \neq d_{G'}(z)$). Now
  \begin{eqnarray*}
    \frac{j_{G'}(y,z)}{j_G(y,z)} & = & \frac{\log \left( 1+\frac{|y-z|}{\min \{ d_{G'}(y),d_{G'}(z) \}} \right)}{\log \left( 1+\frac{|y-z|}{\min \{ d_G(y),d_G(z) \}} \right)} =\frac{\log \left( 1+\frac{|y-z|}{\min \{ |y-x|,d_G(z) \}} \right)}{\log \left( 1+\frac{|y-z|}{\min \{ d_G(y),d_G(z) \}} \right)}.
  \end{eqnarray*}
  Let us assume that $d_G(z) \leq d_G(y)$. If $d_G(z) \leq |y-x|$ then $j_{G'}(y,z) = j_G(y,z)$ and there is nothing to prove. If $d_G(z) \geq |y-x|$ then
  \begin{eqnarray*}
    \frac{j_{G'}(y,z)}{j_G(y,z)} & = & \frac{\log \left( 1+\frac{|y-z|}{|y-x|} \right)}{\log \left( 1+\frac{|y-z|}{d_G(z)} \right)} \leq \frac{\log \left( 1+\frac{|y-z|}{td_G(x)} \right)}{\log \left( 1+\frac{|y-z|}{d_G(z)} \right)}\\
    & \leq & \frac{\log \left( 1+\frac{|y-z|}{td_G(x)} \right)}{\log \left( 1+\frac{|y-z|}{d_G(y)} \right)}\leq \frac{\log \left( 1+\frac{|y-z|}{td_G(x)} \right)}{\log \left( 1+\frac{|y-z|}{|y-x|+d_G(x)} \right)}.
  \end{eqnarray*}
If $|x-y| \leq d_G(x)$ then we have by Lemma \ref{loglemma}
\begin{eqnarray*}
    \frac{j_{G'}(y,z)}{j_G(y,z)} & \leq & \frac{\log \left( 1+\frac{|y-z|}{t d_G(x)} \right)}{\log \left( 1+\frac{|y-z|}{2d_G(x)} \right)} \leq \lim_{|y-z|/d_G(x) \to 0} \frac{\log \left( 1+\frac{|y-z|}{t d_G(x)} \right)}{\log \left( 1+\frac{|y-z|}{2d_G(x)} \right)}\\
    & \leq & \lim_{|y-z|/d_G(x) \to 0} \frac{2+\frac{|y-z|}{d_G(x)}}{t+\frac{|y-z|}{d_G(x)}} = \frac{2}{t}.
  \end{eqnarray*}
If $|x-y|\geq d_G(x)$ again by Lemma \ref{loglemma}
\begin{eqnarray*}
    \frac{j_{G'}(y,z)}{j_G(y,z)} & \leq & \frac{\log \left( 1+\frac{|y-z|}{|y-x|} \right)}{\log \left( 1+\frac{|y-z|}{2|y-x|} \right)} \leq \lim_{|y-z|/|y-x| \to 0} \frac{\log \left( 1+\frac{|y-z|}{|y-x|} \right)}{\log \left( 1+\frac{|y-z|}{2|y-x|} \right)}\\
    & \leq & \lim_{|y-z|/|y-x| \to 0} \frac{2+\frac{|y-z|}{|y-x|}}{1+\frac{|y-z|}{|y-x|}} = 2.
  \end{eqnarray*}
  Let us then assume $d_G(y) \leq d_G(z)$. Now $d_G(y) \neq d_{G'}(y)$ implies $|y-x| < d_G(y)$ and thus
  \begin{equation}\label{jmetricformula}
    \frac{j_{G'}(y,z)}{j_G(y,z)} = \frac{\log \left( 1+\frac{|y-z|}{|y-x|} \right)}{\log \left( 1+\frac{|y-z|}{d_G(y)} \right)} \leq \frac{\log \left( 1+\frac{|y-z|}{|y-x|} \right)}{\log \left( 1+\frac{|y-z|}{|y-x|+d_G(x)} \right)}.
  \end{equation}
  If $|x-y| \leq d_G(x)$ we have by \eqref{jmetricformula} and Lemma \ref{loglemma}
  \begin{eqnarray*}
    \frac{j_{G'}(y,z)}{j_G(y,z)} & \leq & \frac{\log \left( 1+\frac{|y-z|}{t d_G(x)} \right)}{\log \left( 1+\frac{|y-z|}{2d_G(x)} \right)} \leq \lim_{|y-z|/d_G(x) \to 0} \frac{\log \left( 1+\frac{|y-z|}{t d_G(x)} \right)}{\log \left( 1+\frac{|y-z|}{2d_G(x)} \right)}\\
    & \leq & \lim_{|y-z|/d_G(x) \to 0} \frac{2+\frac{|y-z|}{d_G(x)}}{t+\frac{|y-z|}{d_G(x)}} = \frac{2}{t}.
  \end{eqnarray*}
  If $d_G(x) \leq |x-y|$ we have by \eqref{jmetricformula} and Lemma \ref{loglemma}
  \begin{eqnarray*}
    \frac{j_{G'}(y,z)}{j_G(y,z)} & \leq & \frac{\log \left( 1+\frac{|y-z|}{|y-x|} \right)}{\log \left( 1+\frac{|y-z|}{2|y-x|} \right)} \leq \lim_{|y-z|/|y-x| \to 0} \frac{\log \left( 1+\frac{|y-z|}{|y-x|} \right)}{\log \left( 1+\frac{|y-z|}{2|y-x|} \right)}\\
    & \leq & \lim_{|y-z|/|y-x| \to 0} \frac{2+\frac{|y-z|}{|y-x|}}{1+\frac{|y-z|}{|y-x|}} = 2.
  \end{eqnarray*}

  Let us then assume $d_G(y) \neq d_{G'}(y)$ and $d_G(z) \neq d_{G'}(z)$. Now we may assume by symmetry that $|y-x| \leq |z-x|$ and thus
  \begin{eqnarray*}
    \frac{j_{G'}(y,z)}{j_G(y,z)} & = & \frac{\log \left( 1+\frac{|y-z|}{|y-x|} \right)}{\log \left( 1+\frac{|y-z|}{\min \{ d_G(y),d_G(z) \}} \right)} \leq \frac{\log \left( 1+\frac{|y-z|}{|y-x|} \right)}{\log \left( 1+\frac{|y-z|}{|y-x| + d_G(x)} \right)}
  \end{eqnarray*}
  and this is exactly the same as \eqref{jmetricformula} so we know that $$\frac{j_{G'}(y,z)}{j_G(y,z)}\leq \frac{2}{t}.$$

  Putting all this together gives us $c_j = \frac{2}{t}$.

  Let now $m = p$. If $d_G(y) = d_{G'}(y)$ and $d_G(z) = d_{G'}(z)$, then there is nothing to prove as $p_{G'}(y,z) = p_G(y,z)$ and we can choose $c_p=1$. We consider next two cases: $d_G(y) \neq d_{G'}(y)$, $d_G(z) = d_{G'}(z)$ and $d_G(y) \neq d_{G'}(y)$, $d_G(z) \neq d_{G'}(z)$.

  Let us assume $d_G(y) \neq d_{G'}(y)$ and $d_G(z) = d_{G'}(z)$ (or by symmetry we could as well assume $d_G(y) = d_{G'}(y)$ and $d_G(z) \neq d_{G'}(z)$). Now
  \begin{eqnarray*}
    \frac{p_{G'}^2(y,z)}{p_G^2(y,z)} & = & \frac{|y-z|^2+4d_G(y)d_G(z)}{|y-z|^2+4d_{G'}(y)d_{G'}(z)} = \frac{|y-z|^2+4d_G(y)d_G(z)}{|y-z|^2+4 |y-x| d_G(z)}\\
    & \leq & \frac{|y-z|^2+4 (|x-y|+d_G(x)) d_G(z)}{|y-z|^2+4 |y-x| d_{G}(z)}\\
    & = & 1+ \frac{4d_G(x)d_G(z)}{|y-z|^2+4 |y-x| d_{G}(z)} \leq 1+ \frac{4d_G(x)d_G(z)}{0+4td_{G}(x)d_{G}(z)}\\
    & = & 1+\frac{1}{t}.
  \end{eqnarray*}

  Let us then assume $d_G(y) \neq d_{G'}(y)$ and $d_G(z) \neq d_{G'}(z)$. Now
  \begin{eqnarray*}
    \frac{p_{G'}^2(y,z)}{p_G^2(y,z)} & = & \frac{|y-z|^2+4d_G(y)d_G(z)}{|y-z|^2+4d_{G'}(y)d_{G'}(z)} = \frac{|y-z|^2+4d_G(y)d_G(z)}{|y-z|^2+4 |y-x| |z-x|}\\
    & \leq & \frac{|y-z|^2+4 (|x-y|+d_G(x)) (|x-z|+d_G(x))}{|y-z|^2+4 |y-x| |z-x|}\\
    & = & 1+ \frac{4(|x-y|d_G(x) + |x-z|d_G(x) + d_G(x)^2)}{|y-z|^2+4 |y-x| |z-x|}\\
    &\leq & 1+\frac{ 4( |x-y|d_G(x) + |x-z|d_G(x) + d_G(x)^2) }{4 |y-x| |z-x|}\\
    &=& 1+\frac{|x-y|d_G(x)}{|y-x| |z-x|} +\frac{|x-z|d_G(x)}{|y-x| |z-x|}+\frac{d_G(x)^2}{|y-x| |z-x|}\\
    &\leq & 1+\frac{|x-y|d_G(x)}{|y-x| t d_G(x)}+\frac{|x-z|d_G(x)}{t d_G(x) |z-x|}+\frac{d_G(x)^2}{t d_G(x) t d_G(x)}\\
    & = & 1+\frac{2}{t}+\frac{1}{t^2} = 1+\frac{2t+1}{t^2}.
  \end{eqnarray*}
Combining the cases we obtain $c_p = \frac{t+1}{t}$.

Let us finally consider the case $m = s$. Now
\[
  \frac{s_{G'}(y,z)}{s_G(y,z)} = \frac{\inf_{u \in \partial G} |y-u|+|u-z|}{\inf_{u \in \partial G'} |y-u|+|u-z|}
\]
and if the infimum in the denominator is obtained at a point $u \in \partial G$, then there is nothing to prove as $s_{G'}(y,z) = s_G(y,z)$ and we can choose $c_s=1$. If this is not the case, then
\begin{eqnarray*}
  \frac{s_{G'}(y,z)}{s_G(y,z)} & = & \frac{\inf_{u \in \partial G} |y-u|+|u-z|}{\inf_{u \in \partial G'} |y-u|+|u-z|} = \frac{\inf_{u \in \partial G} |y-u|+|u-z|}{|y-x|+|x-z|}\\
  & \leq & \frac{|x-y|+d_G(x)+|x-z|+d_G(x)}{|y-x|+|x-z|} = 1+\frac{2d_G(x)}{|y-x|+|x-z|}\\
  & \leq & 1+\frac{2d_G(x)}{2td_G(x)} = 1+\frac{1}{t}
\end{eqnarray*}
and we can choose $c_s = 1+\frac{1}{t}$.

We see easily that $c_j, c_p, c_s \to 2$ as $t \to 1$. We show next that the constants $c_j$, $c_p$ and $c_s$ are best possible. In all three cases we consider $G=\mathbb{R}^n \setminus \{ 0 \}$.

We start with the case $m=j$. Let $a>0$. For points $x=e_1$, $y=(1+t)e_1$ and $z=(1+t+a)e_1$ we have
\[
  \frac{j_{G'}(y,z)}{j_G(y,z)} = \frac{\log \left( 1+\frac{a}{t} \right)}{\log \left( 1+\frac{a}{1+t} \right)}
  \]
  and
  \[
     \frac{j_{G'}(y,z)}{j_G(y,z)} \to \frac{\log (1+a)}{\log (1+\frac{a}{2})}
   \]
   as $t \to 1$. The asymptotic behavior is clear since
   \[
     \frac{\log (1+a)}{\log (1+\frac{a}{2})} \to 2
    \]
as $a \to 0$.

We next consider the case $m=p$. Let $a \in (0,t]$. For points $x=e_1$, $y=(1+\sqrt{t^2-a^2})e_1+a e_2$ and $z=(1+\sqrt{t^2-a^2})e_1-a e_2$ we have $|y-z|=2a$ and
\[
  \frac{p^2_{G'}(y,z)}{p^2_G(y,z)} = \frac{|y-z|^2+4d_G(y)d_G(z)}{|y-z|^2+4d_{G'}(y)d_{G'}(z)} = \frac{4a^2+4 \left( a^2+ \left( 1+\sqrt{t^2-a^2} \right)^2 \right)}{4a^2+4t^2}.
\]
Now
\[
  \frac{p^2_{G'}(y,z)}{p^2_G(y,z)} \to \frac{4a^2+4 \left( a^2+ \left( 1+\sqrt{1-a^2} \right)^2 \right)}{4a^2+4} = \frac{4a^2+8+8 \sqrt{1-a^2}}{4a^2+4}
\]
as $t \to 1$ and
\[
  \frac{4a^2+8+8 \sqrt{1-a^2}}{4a^2+4} \to 4
\]
as $a \to 0$.

We finally consider the case $m=s$. Let $a \in (0,t]$. For points $x=e_1$, $y=(1+\sqrt{t^2-a^2})e_1+a e_2$ and $z=(1+\sqrt{t^2-a^2})e_1-a e_2$ we have $|y-z|=2a$ and
\begin{eqnarray*}
  \frac{s_{G'}(y,z)}{s_G(y,z)} & = & \frac{\frac{2a}{2t}}{\frac{2a}{2 \sqrt{a^2+ \left( 1+\sqrt{t^2-a^2} \right)^2 }}} = \frac{\sqrt{a^2+ \left( 1+\sqrt{t^2-a^2} \right)^2 }}{t}\\
  & \to & \sqrt{a^2+ \left( 1+\sqrt{1-a^2} \right)^2 }
\end{eqnarray*}
as $t \to 1$ and
\[
  \sqrt{a^2+ \left( 1+\sqrt{1-a^2} \right)^2 } = \sqrt{2+2\sqrt{1-a^2}} \to 2
\]
as $a \to 0$.
\end{proof}

We show next that Theorem \ref{riku4} does not work for the visual angle metric $v$. Let $G = \mathbb{R}^n \setminus \{ 0 \}$ and $x=e_1$. Now for $y = \frac{e_1}{2}$ and $z=2e_1$ we have $v_G(y,z) = 0$ an $v_{G \setminus \{ x \}}(y,z) = \pi$.

\medskip

{\bf Proof of Theorem \ref{mainB}}
  The assertion follows from Theorems \ref{riku3} and \ref{riku4}.\hfill $\square$

\section{Smoothness of $s$-disks with small radii}


In this section, we will consider the smoothness of triangular ratio metric balls in equilateral triangles and rectangles in $\mathbb{R}^2$.
Let $T_{\frac{\pi}{6},2}$ denote the equilateral triangle with vertices $(0,0)$, $(\sqrt{3},1)$, $(\sqrt{3},-1)$, and $R_{a,b}$ denote the rectangle with vertices $(a,b)$, $(a,-b)$, $(-a,b)$, $(-a,-b)$, where $a\geq b>0$.

\begin{figure}[h]
\begin{center}
     \includegraphics[width=8cm]{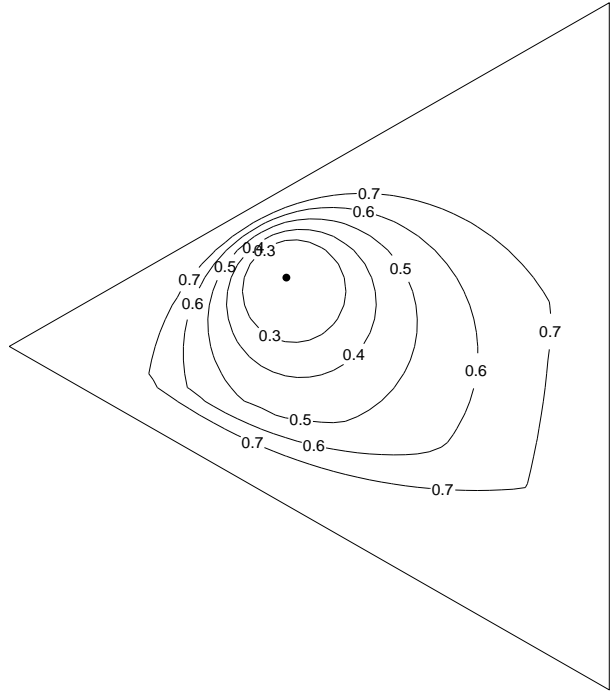}
\caption{Triangular ratio metric balls $ B_{s_{G}}(x,r) $ in $T_{\frac{\pi}{6},2}$. }
\end{center}
    \end{figure}

\begin{figure}[h]
\begin{center}
     \includegraphics[width=10cm]{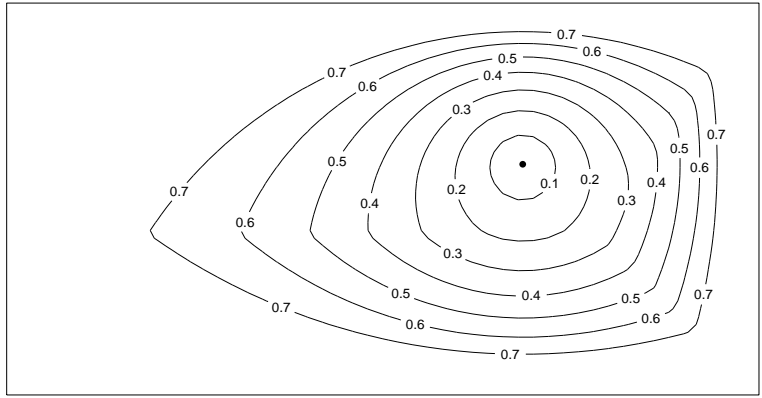}
\caption{Triangular ratio metric balls $ B_{s_{G}}(x,r) $ in $R_{a,b}$. }
\end{center}
    \end{figure}
\medskip

\begin{lem} \label{riku2}
  Let $P \subset \mathbb{R}^2$ be a polygon with inner angles less than or equal to $\pi$ and suppose that there are half planes $H_1$, $H_2$, \dots , $H_n$ such that
\[
  P = \bigcap_{i=1}^n H_i.
\]
  Then for $x \in P$ and $r > 0$ we have
\[
  B_{s_P}(x,r) = \bigcap_{i=1}^n B_{s_{H_i}}(x,r).
\]
\end{lem}
\begin{proof}
  Follows from \cite[Lemma 5.4]{hklv}.
\end{proof}

\medskip
{\bf Proof of Theorem \ref{thm1.8}}
Denote by the lines
$l_1: y_2=\frac{\sqrt{3}}{3}y_1$,
$l_2: y_2=-\frac{\sqrt{3}}{3}y_1$, and
$l_3: y_1=\sqrt{3}$.
For any point $x\in G=T_{\frac{\pi}{6},2}$ and $r\in (0,1)$, by Lemma \ref{riku2}, we have
$$B_{s_{G}}(x,r)=\cap_{i=1}^{3}B_{i},$$
where $B_{i}$ is the corresponding triangular ratio metric ball $B_{s_{G_{i}}}(x,r)$, and $G_{i}$ is the half plane with boundary line $l_{i}.$
By elementary computation, we have that
\begin{align*}
B_{1}:&\left\{ y: \left(y_{1}-\dfrac{(2-r^{2})x_{1}-\sqrt{3}r^{2}x_{2}}{2(1-r^{2})}\right)^{2}+\left(y_{2}-\dfrac{(2+r^{2})x_{2}-\sqrt{3}r^{2}x_{1}}{2(1-r^{2})}\right)^{2}\right.  \\
&\;\;\;\;\;\;\;\;\;\left.<\dfrac{r^{2}(x_{1}-\sqrt{3}x_{2})^{2}}{(1-r^{2})^{2}}\right\},
\end{align*}
\begin{align*}
B_{2}:&\left\{ y: \left(y_{1}-\dfrac{(2-r^{2})x_{1}+\sqrt{3}r^{2}x_{2}}{2(1-r^{2})}\right)^{2}+\left(y_{2}-\dfrac{(2+r^{2})x_{2}+\sqrt{3}r^{2}x_{1}}{2(1-r^{2})}\right)^{2}\right.  \\
&\;\;\;\;\;\;\;\;\;\left.<\dfrac{r^{2}(x_{1}+\sqrt{3}x_{2})^{2}}{(1-r^{2})^{2}}\right\},
\end{align*}
and
\begin{align*}
B_{3}:&\left\{ y: \left(y_{1}-\dfrac{x_{1}-2\sqrt{3}r^{2}+x_{1}r^{2}}{ 1-r^{2} }\right)^{2}+\left(y_{2}- x_{2} \right)^{2} <\dfrac{4r^{2}(x_{1}-\sqrt{3} )^{2}}{(1-r^{2})^{2}}\right\}.
\end{align*}
Hence, $B_{s_{G}}(x,r)$ is smooth if and only if $B_{s_{G}}(x,r)$ is one of the above three balls. It is known that $\mathbb{B}^2(a,r_1)\subset \mathbb{B}^2(b,r_2)$ is equivalent to $|a-b|\leq r_2-r_1$.
Then, by calculations, we have that for any point $x=(x_1,x_2
)\in G$, $B_{1}\subset B_{2}$ and $B_{1}\subset B_{3}$ is equivalent to
$$0<r\leq \dfrac{2x_{2}}{\sqrt{x_{1}^{2}+x_{2}^{2}}},\;\;\text{and}\;\;0<r\leq \dfrac{x_{2}-\sqrt{3}x_{1}+2}{\sqrt{(\sqrt{3}-x_{1})^{2}+(1-x_{2})^{2}}};$$
$B_{2}\subset B_{1}$ and $B_{2}\subset B_{3}$ is equivalent to
$$0<r\leq -\dfrac{2x_{2}}{\sqrt{x_{1}^{2}+x_{2}^{2}}},\;\;\text{and}\;\;0<r\leq \dfrac{-x_{2}-\sqrt{3}x_{1}+2}{\sqrt{(\sqrt{3}-x_{1})^{2}+(1+x_{2})^{2}}};$$
$B_{3}\subset B_{1}$ and $B_{3}\subset B_{2}$ is equivalent to
 $$0<r\leq \dfrac{\sqrt{3}x_{1}-x_{2}-2}{\sqrt{(\sqrt{3}-x_{1})^{2}+(1-x_{2})^{2}}}\;\;\text{and}\;\;0<r\leq \dfrac{\sqrt{3}x_{1}+x_{2}-2}{\sqrt{(\sqrt{3}-x_{1})^{2}+(1+x_{2})^{2}}}.$$
That is for any point $x\in T_{\frac{\pi}{6},2}$, $0<r<1$, $B_{s_{G}}(x,r)$ is smooth if and only if
$$0<r\leq \min \left \{\dfrac{2|x_{2}|}{\sqrt{x_{1}^{2}+x_{2}^{2}}},\dfrac{|x_2|-\sqrt{3}x_1+2 }{\sqrt{(x_1-\sqrt{3})^{2}+(1-|x_2|)^{2}}} \right\},$$ or
$$0<r\leq  \dfrac{\sqrt{3}x_1-2-|x_{2}|}{\sqrt{(x_1-\sqrt{3})^{2}+(1-|x_2|)^{2}}} .$$
Obviously, for $x_2=0$ and $0<x_1 \leq \dfrac{2\sqrt{3}}{3}$, or $|x_2|=\sqrt{3}x_1-2$, $ B_{s_{G}}(x,r) $ cannot be smooth.

For the case $G=R_{a,b}$, let $l_1: y_{2}=b$, $l_2: y_{1}=a$, $l_3: y_{2}=-b$,
and $l_4: y_{1}=-a$.
For any point $x \in R_{a,b}$, and $r\in (0,1)$, it follows from Lemma \ref{riku2} that $$B_{s_{G}}(x,r)=\cap_{i=1}^{4}B_{i},$$
where $B_{i}$ is the corresponding triangular ratio metric ball $B_{s_{G_{i}}}(x,r)$, and $G_{i}$ is the half plane with boundary line $l_{i}$.
For any point $x\in R_{a,b}$, it follows from elementary computation that
\begin{align*}
B_{1}:&\left\{ y: \left(y_{1}-x_{1}\right)^{2}+\left(y_{2}-\dfrac{x_{2}+r^{2}x_{2}-2br^{2}}{1-r^{2}}\right)^{2}<\dfrac{4r^{2}(b-x_{2})^{2}}{(1-r^{2})^{2}}\right\},
\end{align*}
\begin{align*}
B_{2}:&\left\{ y: \left(y_{1}-\dfrac{x_{1}+r^{2}x_{1}-2ar^{2}}{1-r^{2}}\right)^{2}+\left(y_{2}-x_{2}\right)^{2}<\dfrac{4r^{2}(a-x_{1})^{2}}{(1-r^{2})^{2}}\right\},
\end{align*}
\begin{align*}
B_{3}:&\left\{ y: \left(y_{1}-x_{1}\right)^{2}+\left(y_{2}-\dfrac{x_{2}+r^{2}x_{2}+2br^{2}}{1-r^{2}}\right)^{2}<\dfrac{4r^{2}(b+x_{2})^{2}}{(1-r^{2})^{2}}\right\},
\end{align*}
and
\begin{align*}
B_{4}:&\left\{ y: \left(y_{1}-\dfrac{x_{1}+r^{2}x_{1}+2ar^{2}}{1-r^{2}}\right)^{2}+\left(y_{2}-x_{2}\right)^{2}<\dfrac{4r^{2}(a+x_{1})^{2}}{(1-r^{2})^{2}}\right\}.
\end{align*}
For $1\leq i\leq 4$, let $R_{i}$ denote the radius of $B_{i}$.
If $x_{2}>0$, then $R_{3}\geq R_{1}$. By calculations,
$B_{s_{G}}(x,r)=B_{1}$ is equivalent to
$$0<r\leq \min\left\{\dfrac{x_{2}}{b},\dfrac{(a-x_{1})-(b-x_{2})}{\sqrt{(a-x_{1})^{2}+(b-x_{2})^{2}}},\dfrac{(a+x_{1})-(b-x_{2})}{\sqrt{(a+x_{1})^{2}+(b-x_{2})^{2}}}\right\}.$$
If $x_{2}<0$, then $R_{1}\geq R_{3}$. By calculations,
$B_{s_{G}}(x,r)=B_{3}$ is equivalent to
$$0<r\leq \min\left\{-\dfrac{x_{2}}{b},\dfrac{(a-x_{1})-(b+x_{2})}{\sqrt{(a-x_{1})^{2}+(b+x_{2})^{2}}},\dfrac{(a+x_{1})-(b+x_{2})}{\sqrt{(a+x_{1})^{2}+(b+x_{2})^{2}}}\right\}.$$
If $x_{1}>0$, then $R_{4}\geq R_{2}$. By calculations,
$B_{s_{G}}(x,r)=B_{2}$ is equivalent to
$$0<r\leq \min\left\{\dfrac{x_{1}}{a},\dfrac{(b-x_{2})-(a-x_{1})}{\sqrt{(a-x_{1})^{2}+(b-x_{2})^{2}}},\dfrac{(b+x_{2})-(a-x_{1})}{\sqrt{(a-x_{1})^{2}+(b+x_{2})^{2}}}\right\}.$$
If $x_{1}<0$, then $R_{2}\geq R_{4}$. By calculations,
$B_{s_{G}}(x,r)=B_{4}$ is equivalent to
$$0<r\leq \min\left\{-\dfrac{x_{1}}{a},\dfrac{(b-x_{2})-(a+x_{1})}{\sqrt{(a+x_{1})^{2}+(b-x_{2})^{2}}},\dfrac{(b+x_{2})-(a+x_{1})}{\sqrt{(a+x_{1})^{2}+(b+x_{2})^{2}}}\right\}.$$
That is, for any point $x\in R_{a,b}$, $0<r<1$, $B_{s_{G}}(x,r)$ is smooth if and only if
$$0<r\leq \min \left \{\dfrac{|x_{2}|}{b},\dfrac{(a-|x_1|)-(b-|x_2|)}{\sqrt{(a-|x_1|)^{2}+(b-|x_2|)^{2}}} \right\},$$ or
$$0<r\leq \min \left \{\dfrac{|x_{1}|}{a},\dfrac{(b-|x_2|)-(a-|x_1|)}{\sqrt{(a-|x_1|)^{2}+(b-|x_2|)^{2}}} \right\}.$$
Obviously, for $x_2=0$ and $a-|x_1|\geq b$, or $a-|x_1|=b-|x_2|$,
$ B_{s_{G}}(x,r) $ cannot be smooth.
$\square$


\section{Quasiregular maps and triangular ratio metric}

In this section our goal is to summarize some basic facts  about quasiconformal mappings, following closely \cite{avv}, and \cite{vu},
and to prove Theorems \ref{1m} and \ref{2m}.
We assume that the reader is familiar with the basics of this theory.
Here we adopt the standard definition of $K-$quasiconformality and $K-$quasiregularity from J. V\"{a}is\"{a}l\"{a}'s book \cite{v1}
and from \cite{vu}, respectively.
The first result is a quasiregular counterpart of the
Schwarz lemma. Observe that the result is asymptotically sharp when $K\to 1\,.$

The \emph{Gr\"otzsch ring domain} $R_{G,n}(s)$, $s>1$, is a doubly
connected domain with complementary components
$(\overline{\Bn},[se_1,\infty))$. For its capacity we write

\[
   \gamma_n(s)={\rm cap}R_{G,n}(s)=M(\Delta(\overline{\Bn},[se_1,\infty])).
\]

For $K>0$ we define an increasing homeomorphism
$\varphi_{K,n}:[0,1]\to[0,1]$ with $\varphi_{K,n}(0)=0$,
$\varphi_{K,n}(1)=1$ and
\beq\label{phi}
   \varphi_{K,n}(r)=\dfrac{1}{\gamma_n^{-1}(K\gamma_n(1/r))},\quad 0<r<1.
\eeq

The following important estimates are well known \cite[pp.98-99]{vu}
\beq\label{bd4phi}
r^\alpha\leq\varphi_{K,n}(r)\leq\lambda_n^{1-\alpha}r^\alpha\leq2^{1-1/K}Kr^\alpha,\quad
\alpha=K^{1/(1-n)}\, ,
\eeq
\beq\label{bd4phi2}
2^{1-K}K^{-K}r^\beta\leq\lambda_n^{1-\beta}r^\beta\leq\varphi_{1/K,n}(r)\leq
r^\beta,\quad \beta=1/\alpha,
\eeq
where $K\geq1, r\in(0,1)\,,$ and the constant $\lambda_n\in[4,2e^{n-1})$
is the so-called \emph{Gr\"otzsch ring constant}. In particular,
$\lambda_2=4$.

\begin{thm}\label{qc}
Let $G, D$ be either $\mathbb{B}^n$ or $\mathbb{H}^n$ and $f:G\rightarrow fG\subset D$ be a non-constant $K-$quasiregular mapping and let $\alpha=K_I(f)^{1/(1-n)}.$ Then

\begin{eqnarray*}
  \tanh \left(\frac{1}{2}\rho_{D}(f(x),f(y)) \right) & \leq & \varphi_{K,n} \left( \tanh \left( \frac{1}{2}\rho_G(x,y) \right) \right)\\
  & \leq & \lambda_n^{1-\alpha} \left( \tanh \left( \frac{1}{2}\rho_G(x,y) \right) \right)^{\alpha},
\end{eqnarray*}
for all $x, y\in G$.
\end{thm}
\begin{proof}
Recall that the proof in \cite[Theorem 11.2]{vu} for the case $G= D= \mathbb{B}^n$ was based on the formula
\beq\label{5.5}
\mu_{\mathbb{B}^n}(x,y)=\gamma_n\left(\frac{1}{\tanh{\frac{\rho_{\mathbb{B}^n}(x,y)}{2}}}\right),\quad x, y\in\mathbb{B}^n.
\eeq
and the transformation rule of the metric $\mu_{\mathbb{B}^n}$ under quasiregular maps. The same proof also works for the present general case as soon as we prove that the formula \eqref{5.5} also holds for the case of $\mathbb{H}^n$.
For this purpose we use the invariance of $\mu_{\mathbb{B}^n}$ under a M\"obius transformation $h:\mathbb{H}^n\rightarrow \mathbb{B}^n$ to conclude by \eqref{5.5} that for $x, y \in \mathbb{H}^n$
\bequu
\mu_{\mathbb{H}^n}(x,y) &=& \mu_{\mathbb{B}^n}(h(x),h(y))\\
&=& \gamma_n\left(\frac{1}{\tanh{\frac{\rho_{\mathbb{B}^n}(h(x),h(y))}{2}}}\right)\\
&=& \gamma_n\left(\frac{1}{\tanh{\frac{\rho_{\mathbb{H}^n}(x,y)}{2}}}\right)
\eequu
where in the last step we used the invariance of the hyperbolic metric under the M\"obius transformation $h$, see \cite[(2.21)]{vu}. After these observations the proof goes in the same way as in \cite[Theorem 11.2]{vu}.
\end{proof}

\medskip


\medskip

{\bf Proof of Theorem \ref{1m}}.
(1) Because for all $x, y\in\mathbb{H}^n$,
$$s_{\mathbb{H}^n}(x,y)=\tanh \left( \frac{\rho_{\mathbb{H}^n}(x,y)}{2} \right) ,$$
 by Theorem \ref{qc} the proof follows.

(2) By Theorems \ref{qc}, \ref{rk2} and Lemma \ref{3.7} we have for all $x, y\in\mathbb{B}^n$,
\begin{eqnarray}
s_{\mathbb{B}^n}(f(x),f(y))&\leq& \tanh \left( \frac{\rho_{\mathbb{B}^n}(f(x),f(y))}{2} \right) \nonumber\\
&\leq& \lambda_n^{1-\alpha}\tanh \left( \frac{\rho_{\mathbb{B}^n}(x,y)}{2} \right)^{\alpha} \nonumber\\
&\leq& \lambda_n^{1-\alpha}(2s_{\mathbb{B}^n}(x,y))^{\alpha}\nonumber\\
&=& 2^{\alpha}\lambda_n^{1-\alpha}(s_{\mathbb{B}^n}(x,y))^{\alpha}\,. \nonumber
\end{eqnarray}

(3) Similarly by Theorems \ref{qc} and \ref{rk2} we have for all $x, y\in\mathbb{B}^n$,
\bequu
s_{\mathbb{H}^n}(f(x),f(y))&=& \tanh \left( \frac{\rho_{\mathbb{H}^n}(f(x),f(y))}{2} \right)\\
&\leq& \lambda_n^{1-\alpha}\tanh \left( \frac{\rho_{\mathbb{B}^n}(x,y)}{2} \right)^{\alpha}\\
&\leq& \lambda_n^{1-\alpha}(2s_{\mathbb{B}^n}(x,y))^{\alpha}\\
&=& 2^{\alpha}\lambda_n^{1-\alpha}(s_{\mathbb{B}^n}(x,y))^{\alpha}\,.
\eequu

(4) By Theorems \ref{sp}, \ref{3.7} and \ref{qc} we have for all $x, y\in\mathbb{H}^n$,
\bequu
s_{\mathbb{B}^n}(f(x),f(y))&\leq & \tanh \left( \frac{\rho_{\mathbb{B}^n}(f(x),f(y))}{2} \right)\\
&\leq& \lambda_n^{1-\alpha}\tanh \left( \frac{\rho_{\mathbb{H}^n}(x,y)}{2} \right)^{\alpha}\\
&=& \lambda_n^{1-\alpha}(s_{\mathbb{H}^n}(x,y))^{\alpha}\,.  \quad \square
\eequu
\medskip

\begin{thm}
Let $f:\mathbb{B}^n\rightarrow \mathbb{B}^n$ be a $K-$quasiregular mapping.
Then for $x,y\in \mathbb{B}^n$ we have
\begin{equation}
p_{\mathbb{B}^n}(f(x),f(y))\leq 2^{\alpha}\lambda_n^{1-\alpha}(p_{\mathbb{B}^n}(x,y))^{\alpha},\,~ \alpha=K^{1/(1-n)}.
\end{equation}
\end{thm}

\begin{proof}
By Lemma \ref{3.7}, the proof is similar to the proof of Theorem \ref{1m}.
\end{proof}

\medskip

 By definition \eqref{sm} it is clear that for $x,y\in
G=\mathbb{R}^n\setminus\{0\}$, we have
$$s_G(x,y)=\frac{|x-y|}{|x|+|y|}\,.$$
Recall the following notation from \cite[Section 14]{avv},
$$\eta^*_{K,n}(t)=\sup \left\{ |g(x)|:|x|\leq t, g \in {\mathcal F}_K  \right\} ,$$
$$  {\mathcal F}_K = \{ g:{\mathbb R}^n \to {\mathbb R}^n, g(0)=0, g(e_1)=e_1, g \,\,\,{\rm is} \,\,\, K-{\rm quasiconformal} \}.$$

\medskip

\begin{lem}\cite[14.27]{avv}\label{29}
Let $f:\mathbb{R}^n \to  \mathbb{R}^n$ be a $K-$quasiconformal mapping with $f(\infty)=\infty,$
and let $a,b,c$ be three distinct points in $\mathbb{R}^n$. Then
\begin{eqnarray}
\frac{1}{P_6(n,K)}\left(\frac{|a-c|}{|a-b|+|b-c|}\right)^{\beta}&\leq&\frac{|f(a)-f(c)|}{|f(a)-f(b)|+|f(b)-f(c)|}\nonumber\\ &\leq&\frac{1}{P_5(n,K)}\left(\frac{|a-c|}{|a-b|+|b-c|}\right)^\alpha,\nonumber
\end{eqnarray}
where $\alpha=K^{1/(1-n)}=1/\beta$ and $P_5(n,K)=2^{1-\left(\beta/\alpha\right)}\lambda_n^{1-\beta}/\eta^*_{K,n}(1)$,
$P_6(n,K)=2^{1-\left(\alpha/\beta\right)}\lambda_n^{\beta-1}\eta^*_{K,n}(1)$.
Here $\lambda_n$ is as in Lemma \ref{qc} and $P_5(n,K)\to 1, P_6(n,K) \to 1,$ when $K \to 1\,.$
\end{lem}

\medskip

{\bf Proof of Theorem \ref{2m}}.
By M\"obius invariance of the absolute ratio, the result follows from Lemma \ref{29} if we take $b=f(b)=0\,. \quad \square$

\medskip

\begin{lem}\cite[14.8]{avv}\label{14.8}
For $n\geq 2$ and $K\geq 1$,
$$\eta_{K,n}^*(1)\leq\exp(4K(K+1)\sqrt{K-1}).$$
\end{lem}

\medskip

\begin{cor}
Let $G=\mathbb{R}^n\setminus\{0\}$, and $f:G\to G$ be a $K-$quasiconformal mapping. If $n\geq 2$, $\alpha=K^{1/(1-n)}$, then for $z,w\in G$,
\[
s_{fG}(f(z),f(w))\leq K^K \exp(2(K+1)(K-1)+4K(K+1)\sqrt{K-1})\left(s_G(z,w)\right)^\alpha.
\]
\end{cor}

\begin{proof}
Combining Lemmas \ref{29} and \ref{14.8} and by \cite[Lemma 7.50 (2)]{vu} we see that
\bequu
\frac{1}{P_5(n,K)} &=& \frac{\eta^*_{K,n}(1)}{2^{1-\left(\beta/\alpha\right)}\lambda_n^{1-\beta}}\\
&\leq & \frac{2^{K-1} K^K \eta^*_{K,n}(1)}{2^{1-\left(\beta/\alpha\right)}}\\
&\leq & 2^{\left(\beta/\alpha\right)+K-2} K^K \eta^*_{K,n}(1)\\
&\leq & 2^{\left(\beta/\alpha\right)+K-2} K^K \exp(4K(K+1)\sqrt{K-1})\\
&\leq & K^K \exp(\left(\beta/\alpha\right)+K-2+4K(K+1)\sqrt{K-1})\\
&\leq & K^K \exp(2(K+1)(K-1)+4K(K+1)\sqrt{K-1}).
\eequu

\end{proof}


{\bf Acknowledgements.}
The authors are indebted to the referee for a very valuable set of corrections.
This research was supported by the Academy of Finland, Project 2600066611 and the V\"{a}is\"{a}l\"{a} foundation.
 The research of the first author was supported by CIMO and China Scholarship Council.
 The research of the second author was supported also by CIMO. The research of the third author was supported also by the Marsden Fund, New Zealand.


\end{document}